\documentclass[11pt]{article}

\usepackage{amsfonts}
\usepackage{amssymb}
\usepackage{amsmath,mathrsfs}
\usepackage{amsthm}
\usepackage{latexsym}
\usepackage{layout}

\usepackage{textcomp}
\usepackage{amsbsy}
\usepackage{latexsym}
\usepackage[mathscr]{eucal}
\usepackage{amsfonts}
\usepackage{amssymb}
\usepackage[usenames]{color}

\usepackage[english]{babel}

\theoremstyle{plain}
\begingroup
\newtheorem{theorem}{Theorem}[section]
\newtheorem{lemma}[theorem]{Lemma}

\newtheorem{corollary}[theorem]{Corollary}
\endgroup

\theoremstyle{definition}
\begingroup
\newtheorem{definition}[theorem]{Definition}
\newtheorem{remark}[theorem]{Remark}

\endgroup
 
\theoremstyle{remark}
\begingroup

\endgroup

\mathchardef\emptyset="001F

\numberwithin{equation}{section}

\newcommand{\Lone}{{\mathcal L}^{1}}

\newcommand{\supp}{{\rm supp}}

\newcommand{\R}{{\mathbb R}}

\newcommand{\Rn}{{\R}^n}
\newcommand{\Om}{\Omega}
\newcommand{\Ga}{\Gamma}

\newcommand{\N}{\mathbb N}

\newcommand{\PP}{{\mathcal P}_1}
\newcommand{\WW}{{\mathcal W}_1}

\newcommand{\be}{\begin{equation}}
\newcommand{\ee}{\end{equation}}
\newcommand{\bes}{\begin{eqnarray}}
\newcommand{\ees}{\end{eqnarray}}

\begin{document}
\title{Mean-Field Optimal Control}
\author{
Massimo Fornasier
\footnote{Technische Universit\"at M\"unchen, Fakult\"at Mathematik, Boltzmannstrasse 3
 D-85748, Garching bei M\"unchen, Germany  ({\tt massimo.fornasier@ma.tum.de}). } \, and \,
Francesco Solombrino\footnote{Technische Universit\"at M\"unchen, Fakult\"at Mathematik, Boltzmannstrasse 3
D-85748, Garching bei M\"unchen, Germany  ({\tt francesco.solombrino@ma.tum.de}).}
}
\maketitle

\begin{abstract}
We introduce the concept of {\it mean-field optimal control} which is the rigorous limit process connecting finite
dimensional optimal control problems with ODE constraints modeling multi-agent interactions to an infinite dimensional optimal control problem 
with a constraint given by a PDE of Vlasov-type, governing the dynamics of the probability distribution of interacting agents.
While in the classical mean-field theory one studies the behavior of a large number of small individuals {\it freely interacting} with each other, by 
simplifying the effect of all the other individuals on any given individual by a single averaged effect, we address the situation where the individuals
are actually influenced also by an external {\it policy maker}, and we propagate its effect for the number $N$ of individuals going to infinity. 
On the one hand, from a modeling point of view, we take into account also that the policy maker is constrained to act according to optimal strategies promoting its most parsimonious interaction with 
the group of individuals. This will be realized by considering cost functionals including $L^1$-norm terms penalizing a broadly distributed control
of the group, while promoting its sparsity.  On the other hand, from the analysis point of view, and for the sake of generality, we consider broader classes of convex control penalizations. In order 
to develop this new concept of limit rigorously, we need to carefully combine the classical concept of mean-field limit,
connecting the finite dimensional system of ODE describing the dynamics of each individual of the group to the PDE describing the dynamics of the respective
probability distribution, with the well-known concept of $\Gamma$-convergence to show that optimal strategies for the finite dimensional problems
converge to optimal strategies of the infinite dimensional problem. 
\end{abstract}


\section{Introduction}

Recently there has been a strong development of literature in applied mathematics and physics describing collective behavior of 
multiagent systems \cite{CuckerDong11,cucker-mordecki,CucSma07,GC04, MR2000132,KeMinAuWan02,vicsek}, towards modeling phenomena in biology, such as cell aggregation and motion \cite{CDFSTB03,kese70,KocWhi98,be07},  animal motion \cite{BCCCCGLOPPVZ09,MR2507454,ChuDorMarBerCha07,crpito10,CouFra02,CKFL05,CucSma07,Niw94,PE99,ParVisGru02,Rom96,TonTu95,YEECBKMS09}, human \cite{crpito11,CucSmaZho04,MR2438215} and synthetic agent behavior and interactions, such as 
cooperative robots \cite{ChuHuaDorBer07,LeoFio01,PerGomElo09,SugSan97}. As it is very hard to be exhaustive in accounting all the developments of this very fast growing
field, we refer to \cite{CCH13,cafotove10,viza12}  for recent surveys.\\
Most of these models start from particle-like systems, borrowing a leaf from Newtonian physics, by including fundamental ``social interaction'' forces, such as attraction, repulsion,
self-drive, orientation and alignment etc. within classical systems of 2nd order equations, governing the evolution of the status, usually the spatial motion, of each agent. 
One fundamental goal of these studies is to clarify the relationship between the interplay of such simple binary forces, being the ``first principles'' of
social interaction, and the potential emergence  of a global behavior in the form of specific patterns, as the result of the re-iterated superposition in time and group-wise
of such forces. For patterns we do not necessarily mean steady states, as one considers in the study of crystalline structures, but rather structured evolutions, such
as the formation of flocks or swarms in animal motion. Due to their discrete nature and the direct description of the dynamics in terms of the single agent, such
particle models are also called Individual Based Models. While in some cases, for instance in flocking models \cite{CucSma07,cusm07}, it is possible to describe rather precisely
the mechanism of pattern formation, for most of the models the description of the asymptotic behavior of a very large system of particles can become an impossible task.
A classical way to approach the global description of the system is then to focus on its mean behavior, 
 as in the classical mean-field theory one studies the evolution of a large number of small individuals {\it freely interacting} with each other, by 
simplifying the effect of all the other individuals on any given individual by a single averaged effect. This results in considering the evolution of the particle density distribution in
the state variables, leading to so-called mean-field partial differential equations of Vlasov- or Boltzmann-type \cite{1151.82351}. We refer to \cite{CCH13} and the references therein 
for a recent survey  on some of the most relevant mathematical aspects on this approach to swarm models.
\\
On the one hand, in certain circumstances, the formation of a specific pattern is conditional to the initial datum, being positioned in a corresponding basin 
of attraction, which sometimes can be characterized, see \cite{ChuDorMarBerCha07} for an interesting example of modeling of possible multiple patterns. On the other hand, the choice of the initial condition outside such a basin of attraction does not give
in general any guarantee of stable pattern formation. Thus it is interesting to question whether an ``external player'' or ``policy maker'' can intervene
on the system towards pattern formation, also in those situations where this phenomenon is not the result of independent self-organization.
The intervention may be modeled as an additional control vector field subjected to certain bounds, representing the limitations (in terms of resources, strength etc.)
of the external policy maker. \\
In the recent work \cite{CFPT} the authors investigated, specifically for individual based flocking models of Cucker-Smale type \cite{CucSma07,cusm07}, how {\it sparse} controls 
can be applied in order to always ensure pattern formation, in this case the emergence of consensus. \\
For sparse control we mean that the policy maker intervenes the minimal amount of times on the minimal amount of individual agents. Surprisingly this control
strategy turns out not only to be economical in terms of interactions between the policy maker and the group of agents, but also in terms of enhancing the
rate of convergence to pattern formation.\\
While the work \cite{CFPT} clarified the basis of sparse stabilization and optimal control for individual based models, so far it has not been explored how such 
concepts could be rigorously connected, through a proper limit process for the number N of agents going to infinity, to continuum models, as in the classical
aforementioned mean-field theory of {\it uncontrolled} systems. \\

In this paper we want to pose the foundations of the discrete/finite dimensional - continuum/infinite dimensional limit for $N \to \infty$ of ODE constrained
control problems of the type:
\begin{equation}\label{mainmod}
\left \{
\begin{array}{ll}
\dot x_i = v_i, & \\
\dot v_i = ( H \star \mu_N)(x_i,v_i) + f(t,x_i,v_i), & i=1,\dots N, \quad t \in [0,T],
\end{array}
\right.
\end{equation}
where 
$$
\mu_N = \frac{1}{N} \sum_{j=1}^N \delta_{(x_i,v_i)},
$$
is the empirical atomic measure supported on the agents states $(x_i,v_i) \in \mathbb R^{2 d}$, controlled by the minimizer of the cost functional
\begin{equation}\label{mainmod2}
\mathcal E_\psi^N(f) := \int_{0}^T \int_{\mathbb R^{2d}} \left ( L(x,v,\mu_N) + \psi(f(t,x,v) \right ) d \mu_N(t,x,v) dt,
\end{equation}
where $\psi: \mathbb R^{d} \to [0,+\infty)$ is a nonnegative convex function, $f(t,x,v):\mathbb R \times \mathbb R^{d}\times  \mathbb R^{d}  \to \mathbb R^{d}$ is a Carath{\'e}odory function, being an absolutely continuous control vector
field with a certain sublinear growth in the variables $x$ and $v$, $H: \mathbb R^{d} \to \mathbb R^{d}$ is a sublinear and locally  Lipschitz continuous interaction kernel, and $L: \mathbb R^{2d} \times \mathcal P_1 (\mathbb R^{2d})\to \mathbb R_+$
is a continuous function, modeling the discrepancy of the actual states to the basin of attraction. 
In particular we need that for $\mu_j \to \mu$ narrowly in $\mathcal P_1 (\mathbb R^{2d})$, the set of
probability measures with bounded first moment, then $L(x,v,\mu_j) \to  L(x,v,\mu)$ uniformly with respect to $(x,v)$  on compact sets of $\mathbb R^{2d}$.
\\

A relevant choice for $\psi$, modeling a situation of particular interest, is $\psi(\cdot)=\gamma|\cdot|$, for $\gamma >0$. In this case the minimization of $\mathcal E_\psi^N$ simultaneously promotes a choice of an optimal control $f$, which instantaneously steers the system in the direction of the basin of attraction
as a consequence of the minimization of the term involving $L$, and the {\it sparsity} of $f$ by means of the $L^1$-norm penalization term
\begin{equation}\label{L1normcon}
\int_{\mathbb R^{2d}}| f(t,x,v) |  d \mu_N(x,v) = \frac{1}{N} \sum_{j=1}^N | f(t,x_j,v_j)|.
\end{equation}
With this we mean that $\supp(f(t,\cdot))$ is actually expected to be a ``small set''. The use of (scalar) $\ell^1$-norms to penalize controls dates back to the 60's with the models of linear fuel consumption \cite{crlo65}. 
More recent work in dynamical systems \cite{voma06} resumes again $\ell^1$-minimization emphasizing its sparsifying power. 
Also in optimal control with partial differential equation constraints it became rather popular to use $L^1$-minimization to enforce sparsity of controls \cite{caclku12,clku11,clku12,hestwa12,pive12,st09,wawa11}, for instance
in the modeling of optimal placing of actuators or sensors. Let us also mention the recent related work \cite{prtrzu13} on the optimal design problem for sparse control of wave equations, where a suitable convex relaxation (see \cite[Remark 1]{prtrzu13}) leads to the 
transformation of the problem based on an optimization over characteristic functions of small sets to a problem of minimization over terms of the type \eqref{L1normcon}.
\\

The general system \eqref{mainmod} includes, for instance, the sparsely controlled Cucker-Smale type of models of flocking \cite{CFPT}, obtained by choosing $H(x,v) = a(|x|) v$, where $a \in C^{1}([0,+\infty))$ is a {\it nonincreasing positive function},
and $L(x,v,\mu_N) = |v - ( \int_{\mathbb R^{2d}} w d\mu_N(y,w) ) |^2 = | v - \frac{1}{N} \sum_{j=1}^N v_j|^2$, for which one gets
\begin{equation}\label{CSmod}
\left \{
\begin{array}{ll}
\dot x_i = v_i, & \\
\dot v_i =  \frac{1}{N} \sum_{j=1}^N a(|x_j-x_i|) ( v_j - v_i) + f(t,x_i, v_i), & i=1,\dots N, \quad t \in [0,T],
\end{array}
\right.
\end{equation}
subjected to the optimal control $f$ minimizing the cost functional
\begin{equation}\label{CSmod2}
\mathcal E_\gamma^N(f) := \int_{0}^T \frac{1}{N} \sum_{i=1}^N \left (\left |v_i  - \frac{1}{N} \sum_{j=1}^N v_j \right |^2    + \gamma | f(t,x_i,v_i) | \right ) dt.
\end{equation}
The interested reader can compare \eqref{CSmod}-\eqref{CSmod2} with the sparse optimal control problem studied in \cite[Section 5]{CFPT}.
\\

The main result of this work is to clarify in which sense the finite dimensional solutions of \eqref{mainmod}-\eqref{mainmod2} converges for $N \to \infty$ to a solution of the
PDE constrained problem
\begin{equation}\label{PDEmod}
\frac{\partial \mu}{\partial t} + v \cdot \nabla_x \mu = \nabla_v \cdot \left [ \left ( H \star \mu + f \right ) \mu \right ],
\end{equation}
controlled by the minimizer $f$ of the cost functional
\begin{equation}\label{PDEmod2}
\mathcal E_\psi(f) := \int_{0}^T \int_{\mathbb R^{2d}} \left ( L(x,v,\mu) + \psi( f(t,x,v) ) \right ) d\mu(t,x,v) dt,
\end{equation}
where $\mu:[0,T] \to \mathcal P_1(\mathbb R^{2d})$ is a probability measure valued weak solution to \eqref{PDEmod}. Our arguments will be based on the
combination of the concepts of mean-field limit, using techniques of optimal transport \cite{AGS}, in order to connect \eqref{mainmod} to \eqref{PDEmod},
and $\Gamma$-limit \cite{DM} in order to connect the minimizations of \eqref{mainmod2} and \eqref{PDEmod2}.
Accordingly we call this limit process {\it mean-field optimal control}.\\
 Let us stress that some of the relevant ingredients of our theory were already partially
available in the literature. In particular, the rigorous derivation of the mean-field limits to connect \eqref{mainmod} to \eqref{PDEmod} for situations
where no control is addressed, i.e., when $f \equiv 0$, has been already considered, for instance, in \cite{CanCarRos10}. Nevertheless, although it represents a
minor extension, the situation where a control $f$ is present in the equations, and it has potentially a discontinuous nature in time, requires to generalize
the results in \cite{CanCarRos10} to solutions of Carath{\'e}odory for \eqref{mainmod} \cite{Fil}. We sketch these generalizations in the Appendix for the sake
of completeness. In particular, existence, uniqueness, and stability of weak measure-valued solutions to \eqref{PDEmod} with compactly supported data will be given
in details in Theorem \ref{thm:6} and Theorem \ref{uniq}.\\
Additional tools are certain compactness arguments in $L^q((0,T),W^{1,\infty}_{loc}(\mathbb R^{2d},\mathbb R^d))$ for the derivation of a limit for the controls (Theorems \ref{thm:3} and Corollary \ref{cor:2}), and compactness
arguments in $1$-Wasserstein distance for probability measures in $\mathcal P_1(\mathbb R^{2d})$ in order to derive limits of the empirical measures to weak solutions
of \eqref{PDEmod}. Finally, the optimality conditions for the limit controls will be derived using  lower-semicontinuity arguments for the energy 
$\mathcal E_\psi(f)$ in order to obtain the $\Gamma$-$\liminf$ condition (Theorem \ref{thm:5}), and the construction of solutions to \eqref{PDEmod} in order to define a recovery sequence for the 
$\Gamma$-$\limsup$ condition (Theorem \ref{thm:6}).
\\

Beside the specific novelty of our model, where we considered collective behavior ({\it sparsely}) controlled by an external policy maker restricted by {\it limited} resources, we stress again that the originality  of our analysis stands precisely in the combination of the concepts of mean-field- and $\Gamma$-limits, where the reference topologies
are those of $ \mathcal P_1(\mathbb R^{2d})$ for the solutions and $L^q((0,T),W^{1,\infty}_{loc}(\mathbb R^{2d},\mathbb R^d))$  for the controls. This distinguishes our work from other
conceptually similar approaches where limits of finite dimensional optimal control problems to infinite dimensional control problems are considered.
We refer in particular to two main directions. \\
The first is the discretization of PDE constrained optimal control problems by means, e.g., of finite element methods.
One defines a suitable finite dimensional time and/or space discretization and shows that corresponding finite dimensional optimal control solutions converge to the 
solution of the  PDE constrained optimal control problem. Let us stress that such a type of arguments have been applied mainly for elliptic and
parabolic type of equations, and the tools used are either explicit  a priori Galerkin-type error estimates or adaptive discretizations, driven by a posteriori error 
estimates in classical Sobolev spaces. Without being able to be at all exhaustive in describing the vast literature on this well-established methodology, we refer to a classical reference \cite{glo08} and to the recent survey paper \cite{rave10} and the bibliography therein. \\
In order to encounter Vlasov-type transport equations as \eqref{PDEmod}, one needs to refer to the second main direction conceptually similar
to our approach, i.e., the mean-field games, introduced by Lasry and Lions \cite{lali07}, and independently in the optimal control community under the name
Nash Certainty Equivalence (NCE) within the work \cite{HCM03}, later greatly popularized within consensus problems, 
for instance in \cite{NCM10,NCM11}. The first fundamental difference with our work is that in (mean-field) games, each individual agent is competing freely with
the others towards the optimization of its individual goal, as for instance in the financial market, whereas in our model we are concerned with the optimization of the intervention 
of an external policy maker or coordinator endowed with rather limited resources to help the system to form a pattern, when self-organization
does not realize it autonomously, as it is a case, e.g., in modeling economical policies and government strategies.
Let us stress that in our model we are particularly interested to sparsify the control towards most effective results, and also that such an economical concept does not appear anywhere in the literature
when we deal with mean-field limits of large particle systems.
Secondly in mean-field games the stochastic component plays a relevant role (also for the technical derivation of mean-field limits), while in our deterministic model no stochastic terms are
necessarily requested in order to have sufficient regularization for deriving rigorously the mean-field limit. 
\\

The paper is organized as follows: in Section 2 we introduce the class of control functions and we prove its closedness and compactness properties, and certain lower-semicontinuity results related to the cost functional
\eqref{PDEmod2}. Section 3 is dedicated to the finite dimensional optimal control problem \eqref{mainmod}-\eqref{mainmod2} and its well-posedness. In Section 4 we address both the mean-field limit 
to connect \eqref{mainmod} to \eqref{PDEmod} and the conditions of $\Gamma$-convergence to connect the minimizations of \eqref{mainmod2} and \eqref{PDEmod2}, to eventually conclude with Section 5 where
we state our main mean-field optimal control result, which summarizes all our findings. For the sake of a broad readability of the paper and its self-containedness we also included an Appendix recalling the relevant results on Carath{\'e}odory solutions of ODEs and how
they are related via the method of the characteristics to solutions of \eqref{PDEmod}.

\section{The Space of Admissible Controls}

\subsection{Admissible controls}

Let $d \geq 1$ be the dimensionality of the control output, $n \geq 1$ be the dimensionality of the state variables (later we will consider $n=2d$). 

\begin{definition}\label{def:admcontr}
For a horizon time $T>0$, and an exponent $1\le q<+\infty$ we fix a control bound function $\ell \in L^q(0,T)$. The class of admissible control functions $\mathcal F_\ell([0,T])$
is so defined: $f \in \mathcal F_\ell([0,T])$ if and only if 
\begin{itemize}
\item[(i)] $f: [0,T] \times \mathbb R^n \to \mathbb R^d$ is a Carath{\'e}odory function, 
\item[(ii)] $f(t, \cdot) \in W^{1,\infty}_{loc}(\mathbb R^n, \mathbb R^d)$  for almost every $t \in [0,T]$, and
\item[(iii)] $|f(t,0)| + \operatorname{Lip}(f(t,\cdot),\mathbb R^d) \leq \ell(t)$ for almost every $t \in [0,T]$.
\end{itemize}
\end{definition}
Functions in the class $\mathcal F_\ell([0, T])$ can be also regarded as measurable mappings with values in Banach spaces, as we clarify in the next two remarks.

\begin{remark}\label{rem:a}
Every control function $f\in \mathcal F_\ell([0,T])$ can be identified with a mapping $f:[0, T] \to W^{1,\infty}_{loc}(\Rn;\R^d)$ where $f(t)$ is simply the function taking the value $f(t,x)$ at $x$. Let us show now that $f \in L^q((0,T), W^{1,p}(\Om,\R^d))$ for every open bounded subset $\Om \subset \Rn$ and $1< p <+\infty$, with $q$ being exactly the integrability exponent of $\ell$ . To prove that the mapping is measurable, by separability of $L^q((0,T), W^{1,p}(\Om,\R^d))$, it suffices to show weak measurability.
Since $f$ is Carath{\'e}odory, by density of atomic measures in the weak-$*$ topology of measures, the map $t \to \langle f(t), \mu \rangle$ is measurable for every $\mu \in M_b(\Om,\R^d)$. The former duality pairing is the standard one between continuous functions and measures. This holds now in particular when $\mu$ is a function in $L^{(p^{*})'}(\Om, \R^d)$, with $p^{*}$ the Sobolev exponent. Since $W^{1,p}(\Om,\R^d)$ is densely embedded into $L^{p^{*}}(\Om, \R^d)$, then $L^{(p^{*})'}(\Om, \R^d)$ is densely embedded into the dual space of $W^{1,p}(\Om,\R^d)$ endowed with the weak topology. It follows that $t \to \langle \phi, f(t)\rangle$ is measurable for all $\phi \in (W^{1,p}(\Om,\R^d))'$, as we wanted. Finally, one easily has by the assumptions that $\|f(t)\|_{W^{1,p}(\Om,\R^d)} \in L^q(0, T)$, so that $f \in L^q((0,T), W^{1,p}(\Om,\R^d))$.
\end{remark}

\begin{remark}\label{rem:b}
Conversely, consider a mapping $f:[0, T] \to W^{1,\infty}_{loc}(\Rn,\R^d)$ such that \\$f \in L^q((0,T), W^{1,p}(\Om,\R^d))$ for every open bounded subset $\Om \subset \Rn$ and $1< p <+\infty$, the identification $f(t,x)=f(t)(x)$ for all $x\in \Om$ gives us a Carath{\'e}odory function. It makes then sense to consider the subset $C_{\ell,\Om}$ of $L^q((0,T), W^{1,p}(\Om,\R^d))$ defined by
\begin{equation}\label{convesso}
C_{\ell,\Om}:=\{f\in L^q((0,T), W^{1,p}(\Om,\R^d)): |f(t,0)| + \operatorname{Lip}(f(t,\cdot),\mathbb R^d) \leq \ell(t)\hbox{ for a.e. }t \in [0,T]\}\,.
\end{equation}
It easily turns out that $C_{\ell,\Om}$ is convex. Furthermore, if $f \in C_{\ell,\Om}$ for all $\Om \subset \Rn$, it can be identified with a $f \in \mathcal F_\ell([0,T])$.
\end{remark}
In the following, functions in the class $F_\ell([0, T])$ will be identified with measurable mappings $f:[0, T] \to W^{1,\infty}_{loc}(\Rn,\R^d)$ and vice-versa, according to Remarks \ref{rem:a} and \ref{rem:b}, without further specification.

We also point out some closedness properties of the convex set $C_{\ell, \Om}$ introduced in \eqref{convesso}.

\begin{remark}\label{rem:closedness}
Fix $1<p<+\infty$ and a bounded smooth subset $\Om \subset \Rn$. Take a sequence $(f_j)_{j \in \mathbb N}$ in $C_{\ell,\Om}$
such that $f_j(t)$ converges to $f(t)$ in $W^{1,p}(\Om, \mathbb R^d)$ for a.e.\ $t \in [0,T]$. Then, for a.e.\ $t$ the $W^{1,\infty}(\Om,\R^d)$ norm of $f_j(t)$ is bounded because of the definition of $C_{\ell,\Om}$, so that $f_j(t)$ converges to $f(t)$ weakly-$*$ in $W^{1,\infty}(\Om, \mathbb R^d)$ for a.e.\ $t \in [0,T]$, and
$$
|f(t,0)| + \operatorname{Lip}(f(t,\cdot),\Omega) \leq \liminf_{j \to \infty} |f_j(t,0)| + \operatorname{Lip}(f_j(t,\cdot),\Omega)\,.
$$
It follows that $C_{\ell,\Om}$ is closed with respect to pointwise a.e.\ convergence, and therefore in the \\ $L^q((0,T), W^{1,p}(\Om,\R^d))$ norm topology, since any Cauchy sequence in $L^q((0,T), W^{1,p}(\Om,\R^d))$ has a pointwise a.e.\ converging subsequence. 
Since $C_{\ell, \Om}$ is convex, we deduce from Mazur's Lemma that it is also closed in the weak topology of $L^q((0,T), W^{1,p}(\Om,\R^d))$.
\end{remark}

In the following we shall usually fix a horizon time $T>0$ and denote $\mathcal F_\ell:=\mathcal F_\ell([0,T])$, omitting the time interval. The integrability exponent of $\ell$ will be depending on the cost functional we consider, as we will make precise in Section \ref{sec:finite}. In the case of a cost functional of the type \eqref{CSmod2}, the one we are mainly interested in, we will choose $q=1$.

\subsection{Compactness, closedness, and lower semicontinuity properties}

The following compactness result is a sort of generalization of the Dunford-Pettis theorem \cite[Theorem 1.38]{AFP00} for equi-integrable families
of functions with values in a reflexive and separable Banach spaces. Its derivation is standard, but we include its proof for the
sake of completeness.

\begin{theorem}\label{thm:0}
Let $X$ be a reflexive and separable Banach space. Let $(f_j)_{j \in \mathbb N}$ be a sequence of functions in $L^q((0,T),X)$ with $1\le q <+\infty$. Let us also assume
that there exists a map $m \in L^q(0,T)$ such that $\|f_j(t)\|_X \leq m(t)$ for almost all $t \in [0,T]$. Then there exist 
a subsequence $(f_{j_k})_{k \in \mathbb N}$ and a function $f \in L^q((0,T),X)$ such that 
\begin{equation}\label{weaklim}
\lim_{k \to \infty} \int_0^T \langle \phi(t), f_{j_k}(t, \cdot) -  f(t, \cdot) \rangle dt = 0,
\end{equation}
for all $\phi \in L^{q'}((0,T), X')$, with $q'$ the conjugate exponent of $q$, and
\begin{equation}\label{L1weak}
w-\lim_{k \to \infty} \int_{t_1}^{t_2} f_{j_k}(t) dt = \int_{t_1}^{t_2} f(t) dt, \quad \mbox{for all } t_1.t_2 \in [0,T],
\end{equation}
where the limit is in the sense of the weak topology of $X$ and the integrals are in the sense of Bochner.
\end{theorem}
\begin{proof}
Let us first of all recall that, as $X$ is reflexive, it has the Radon-Nikodym property: in particular, given $F \in \operatorname{AC}([0,T],X)$ there exists
$f \in L^1((0,T),X)$ such that
$$
F(t_2)-F(t_1) = \int_{t_1}^{t_2} f(t) dt, \quad \mbox{for all } t_1,t_2 \in [0,T].
$$
Let us now define $F_j(s) = \int_0^s f_j(t) dt$, and we have 
$$
\| F_j(t_2) - F_j(t_1) \|_X \leq \int_{t_1}^{t_2} m(t) dt, \quad \mbox{for all } t_1,t_2 \in [0,T]. 
$$
Therefore $F_j$ are equi-bounded and equi-absolutely continuous, and by Ascoli-Arzel\`a theorem, there exist a subsequence 
$(F_{j_k})_{k \in \mathbb N}$ and a function $F \in \operatorname{AC}([0,T],X)$ such that 
\begin{equation}\label{weaklimF}
w-\lim_{k \to \infty} ( F_{j_k}(t_2) - F_{j_k}(t_1)) = F(t_2) - F(t_1),
\end{equation}
weakly in $X$ for all $t_1,t_2 \in [0,T]$ (this is a consequence of the separability of $X$). By the aforementioned Radon-Nikodym property, there exists
$f \in L^1((0,T),X)$ such that
$$
F(t_2)-F(t_1) = \int_{t_1}^{t_2} f(t) dt, \quad \mbox{for all } t_1,t_2 \in [0,T].
$$
Moreover, by the weak limit \eqref{weaklimF} and lower-semicontinuity of the norm of $X$,
$$
\| F(t_2)-F(t_1)\|_X \leq \liminf_{k \to \infty} \| F_{j_k}(t_2) - F_{j_k}(t_1) \|_X \leq \int_{t_1}^{t_2} m(t) dt,
$$
hence the modulus of absolute continuity of $F$ is again $\int_{t_1}^{t_2} m(t) dt$ and, by the Lebsegue theorem for
functions with values in Banach spaces \cite[2.9.9]{Fed69}, we obtain $\|f(t)\|_X \leq m(t)$ 
for almost every $t \in [0,T]$. It follows that $f \in L^q((0,T),X)$. As 
$$
w-\lim_{k \to \infty} \int_{t_1}^{t_2} f_{j_k}(t) dt = \int_{t_1}^{t_2} f(t) dt,
$$ 
weakly in $X$ for all $t_1,t_2 \in [0,T]$, we actually have
$$
w-\lim_{k \to \infty} \int_A f_{j_k}(t) dt = \int_A f(t) dt,
$$ 
for all $A$ open subsets of $[0,T]$, and therefore for all Borel subsets $A$ of $[0,T]$. Hence for any simple function
$\phi (t) = \sum_{i=1}^m \phi_i \chi_{A_i}(t)$ where $\phi_i \in X'$, we have $\phi \in L^\infty((0,T),X')$ and
$$
\lim_{k \to \infty} \int_0^T \langle \phi(t), f_{j_k}(t) \rangle dt = \int_0^T \langle \phi(t), f(t) \rangle dt,
$$ 
where now the convergence is of real values, and one concludes the proof by density of such simple functions in $L^{q'}((0,T),X')$
and an application of the dominated convergence theorem taking into account the $q$-integrability of $m$.
\end{proof}

The following local compactness property is fundamental to our analysis.

\begin{theorem}\label{thm:1}
Let $\Omega$ be a bounded, smooth, and open subset of $\mathbb R^n$ and let $1<p < \infty$ and $1\le q<+\infty$. 
Assume that  $(f_j)_{j \in \mathbb N}$ be a sequence of functions  in $L^q((0,T),W^{1,p}(\Omega, \mathbb R^d))$ such that
$$
|f_j(t,0) | + \operatorname{Lip}(f_j(t, \cdot),\Omega) \leq \ell(t) \in L^q(0,T),\quad \mbox{for almost every } t \in [0,T], \mbox{ for all } j \in \mathbb N.
$$
Then there exist a subsequence  $(f_{j_k})_{k \in \mathbb N}$ and a function $f \in L^q((0,T),W^{1,p}(\Omega, \mathbb R^d))$ such that
\begin{equation}\label{weakconv}
w-\lim_{k \to \infty} f_{j_k} = f,
\end{equation}
weakly in $L^q((0,T),W^{1,p}(\Omega, \mathbb R^d))$, and 
\begin{equation}\label{closedness}
|f(t,0) | + \operatorname{Lip}(f(t, \cdot),\Omega) \leq \ell(t), \mbox{ for almost every } t \in [0,T].
\end{equation}
\end{theorem}
\begin{proof}
By an application of Theorem \ref{thm:0} for $X =W^{1,p}(\Omega, \mathbb R^d)$, there exist a subsequence  $(f_{j_k})_{k \in \mathbb N}$ and a function $f \in L^q((0,T),W^{1,p}(\Omega, \mathbb R^d))$ such that
\eqref{weakconv} holds. It remains to show \eqref{closedness}. Defining $C_{\ell, \Om}$ as in \eqref{convesso},   \eqref{closedness} is equivalent to saying that $f \in C_{\ell, \Om}$. The conclusion is therefore immediate, since $C_{\ell, \Om}$ is closed with respect to the weak topology $L^q((0,T),W^{1,p}(\Omega, \mathbb R^d))$ by Remark \ref{rem:closedness}.
\end{proof}
An immediate consequence of the previous theorem is the following weak compactness result in $\mathcal F_\ell$.

\begin{corollary}\label{cor:1}
Let $1<p < \infty$. Assume that  $(f_j)_{j \in \mathbb N}$ be a sequence of functions in $\mathcal F_\ell$ for a given function $\ell \in L^q(0,T)$, $1\le q<+\infty$. Then there exist a subsequence  $(f_{j_k})_{k \in \mathbb N}$ and a function $f\in \mathcal F_\ell$, such that
\begin{equation}\label{locweakconv}
\lim_{k \to \infty} \int_0^T \langle \phi(t), f_{j_k}(t, \cdot) -  f(t, \cdot) \rangle dt = 0,
\end{equation}
for all $\phi \in L^{q'}([0,T], H^{-1,p'}(\mathbb R^n, \mathbb R^d))$ such that $\supp (\psi(t)) \Subset  \Omega$ for all $t \in [0,T]$, where $\Omega$ is a relatively compact set in $\mathbb R^n$. Here the symbol $\langle \cdot, \cdot \rangle$ denotes the duality between $W^{1,p}$ and its dual $H^{-1,p'}$.
\end{corollary}
\begin{proof}
By considering an invading countable sequence $(\Omega_h)_{h \in \mathbb N}$ of bounded, smooth, and open subsets of $\mathbb R^n$ and using a diagonal argument, one shows that 
there exist a subsequence  $(f_{j_k})_{k \in \mathbb N}$ and a function $f \in L^q((0,T),W^{1,p}(\Omega_h, \mathbb R^d))$ such that $w-\lim_{k \to \infty} f_{j_k} = f$ weakly in $L^q((0,T),W^{1,p}(\Omega_h, \mathbb R^d))$ for all $h \in \mathbb N$, and  
$$
|f(t,0) | + \operatorname{Lip}(f(t, \cdot),\mathbb R^n) = \sup_{h \in\mathbb N} |f(t,0) | + \operatorname{Lip}(f(t, \cdot),\Omega_h) \leq \ell(t), \mbox{ for almost every } t \in [0,T].
$$
Hence actually $f \in \mathcal F_\ell$. In order to conclude the validity of \eqref{locweakconv}, it is now sufficient to observe
that if $\Omega$ is relatively compact, then there exists $h \in \mathbb N$ such that $\Omega \subset \Omega_h$.
\end{proof}
\begin{remark}
By duality and Morrey's embedding theorem for $n< p < \infty$, if $\psi(t)$ is actually measure valued, then it is also  $H^{-1,p'}$-valued. This observation is used silently in 
the proofs of the results which follow.
\end{remark}

\begin{remark}
The existence of a subsequence of indexes $j_k$ independent of $t$ so that $f_{j_k}(t,\cdot)$ converges to $f(t,\cdot)$ weakly in $W^{1,p}_{loc}(\Rn,\R^d)$ for all $t \in [0, T]$ is in general false. It suffices to think of the sequence $f_{j}(t,\zeta)=\sin(2\pi jt)\zeta$ which, however, converges to $0$ in the sense of \eqref{locweakconv} as a consequence of the Riemann-Lebesgue Lemma.
\end{remark}

In the following we consider the space $\mathcal P_1(\mathbb R^n)$, consisting of all probability measures on $\mathbb R^n$ with finite
first moment. On this set we shall consider the following distance, called the {\it Monge-Kantorovich-Rubistein distance},
\begin{equation}\label{mkrdist}
\mathcal W_1(\mu,\nu)=\sup \left \{ \left | \int_{\mathbb R^n} \varphi(x) d (\mu-\nu)(x)  \right| : \varphi \in \operatorname{Lip}(\mathbb R^n), \quad \operatorname{Lip}(\varphi) \leq 1 \right \}, 
\end{equation}
where $\operatorname{Lip}(\mathbb R^n)$ is the space of Lipschitz continuous functions on $\mathbb R^n$ and $\operatorname{Lip}(\varphi)$ the Lipschitz constant of a function $\varphi$. 
Such a distance can also be represented in terms of optimal transport plans by Kantorovich duality in the following manner: if we denote $\Pi(\mu,\nu)$ the set of transference plans
between the probability measures $\mu$ and $\nu$, i.e., the set of probability measures on $\mathbb R^n \times \mathbb R^n$ with first and second marginals equal to $\mu$ and
$\nu$ respectively, then we have
\begin{equation}\label{wasserstein}
\mathcal W_1(\mu,\nu) =\inf_{\pi \in \Pi(\mu,\nu)} \left  \{ \int_{\mathbb  R^n \times \mathbb R^n} |x-y| d \pi(x,y) \right \}.
\end{equation}
In the form \eqref{wasserstein} the distance $\mathcal W_1$ is also known as the $1$-Wasserstein distance. 
We refer to \cite{AGS,vi09} for more details.

\begin{theorem}\label{thm:3}
For a given $\ell \in L^1(0,T)$, let $(f_k)_{k \in \mathbb N}$ be a sequence of functions in $\mathcal F_\ell$ converging to $f$ in the sense of \eqref{locweakconv}. Let $\mu_k:[0,T] \to \mathcal P_1(\mathbb R^{n})$ be a sequence of functions taking values
in the probability measures with finite first moment, and $\mu :[0,T] \to \mathcal P_1(\mathbb R^n)$ such that
\begin{equation}\label{unifmom}
\sup_{t \in [0,T]} \int_{\mathbb R^n} |x| d\mu_k(t,x) = M < \infty,
\end{equation}
and
\begin{equation}\label{convwass}
\lim_j \mathcal W_1(\mu_k(t),\mu(t)) = 0, \mbox{ for all } t \in [0,T].
\end{equation}
Then
\begin{equation}\label{compmu}
\lim_k \int_0^{\hat t} \langle \varphi, f_k(t,\cdot) \mu_k(t) \rangle dt = \int_0^{\hat t} \langle \varphi, f(t,\cdot) \mu(t) \rangle dt,
\end{equation}
for all $\varphi \in C_c^1(\mathbb R^n, \mathbb R^d)$ and for all $\hat t \in [0,T]$.
\end{theorem}
\begin{proof}
Let us again fix $p>n$. Once we fix $\varphi \in C_c^1(\mathbb R^n, \mathbb R^d)$,  by the assumption $f_k \in \mathcal F$ we have
\begin{eqnarray*}
\operatorname{Lip}(\varphi f_k(t,\cdot)) &\leq& \ell(t) \|\varphi\|_\infty+ \|\nabla \varphi\|_\infty \| f_k(t,\cdot) \|_{L^\infty(B(0,R))} \nonumber \\
&\leq& \ell(t) \Big (\|\varphi\|_\infty+ \|\nabla \varphi\|_\infty (1+R) \Big),
\end{eqnarray*}
where $R>0$ is such that $\supp(\varphi) \Subset B(0,R)$. It follows that
\begin{eqnarray}
&& \limsup_k \left | \int_0^{\hat t} \langle \varphi, f_k(t,\cdot) \mu_k(t) \rangle dt- \int_0^{\hat t} \langle \varphi, f_k(t,\cdot) \mu(t) \rangle dt \right | \nonumber \\
&\leq& \Big (\|\varphi\|_\infty+ \|\nabla \varphi\|_\infty (1+R) \label{equiLip} \Big) \limsup_k  \int_0^T  \ell(t) \mathcal W_1(\mu_k(t),\mu(t)) dt. \label{vanishint}
\end{eqnarray}
From \eqref{convwass} we have the vanishing pointwise convergence almost everywhere of the latter integrand 
$$\ell(t)  \mathcal W_1(\mu_k(t),\mu(t)) \to 0, \quad k \to \infty.$$ 
Moreover, by recalling the definition \eqref{wasserstein} and using the uniform first moment condition \eqref{unifmom}, we obtain
\begin{equation}\label{boundw1}
\mathcal W_1(\mu_k(t),\mu(t)) \leq \int_{\mathbb  R^n \times \mathbb R^n} |x-y| d\mu_k(t,x) d\mu(t,y) \leq M + \int_{\mathbb R^n} |y|d\mu(y),
\end{equation}
uniformly with respect to $t \in [0,T]$. Hence, by dominated convergence theorem applied to \eqref{vanishint} we finally have 
$$
\lim_{k \to \infty} \left | \int_0^{\hat t} \langle \varphi, f_k(t,\cdot) \mu_k(t) \rangle dt- \int_0^{\hat t} \langle \varphi, f_k(t,\cdot) \mu(t) \rangle dt \right | =0.
$$
Therefore, it is sufficient now to show that
$$
\lim_k  \int_0^{\hat t} \langle \varphi, f_k(t,\cdot) \mu(t) \rangle dt =\int_0^{\hat t} \langle \varphi, f(t,\cdot) \mu(t) \rangle dt.
$$
This follows immediately from Corollary \ref{cor:1}, because $t \to \varphi \mu(t) $ is a map belonging to the space $L^\infty([0,\hat t], H^{-1,p'}(\mathbb R^n, \mathbb R^d))$ with 
uniform compact support, since $\varphi$ is such.
\end{proof}
As in the definition of a weak solution of the equation \eqref{PDEmod} the role of $\varphi$ is played actually by $\nabla_v \varphi$ (where $\varphi \in C_c^1(\mathbb R^n)$), see formula \eqref{seconda} in the proof of Theorem \ref{thm:5}, we need
to extend the validity of Theorem \ref{thm:3} as follows.
\begin{corollary}\label{cor:2}
The statement of Theorem \ref{thm:3} actually holds also for $\varphi \in C_c^0(\mathbb R^n,\mathbb R^d)$.
\end{corollary}
\begin{proof}
Let us notice that 
\begin{equation}\label{unifcont}
|\langle \varphi, f_k(t,\cdot) \mu_k(t) \rangle| \leq \|\varphi\|_\infty \| f_k(t,\cdot) \|_{L^\infty(B(0,R))} \leq  \|\varphi\|_\infty \ell(t) (1+R),
\end{equation}
where $R$ is such that $\supp(\varphi) \Subset B(0,R)$. By uniform approximation by functions in $C_c^1(\mathbb R^n,\mathbb R^d)$, 
the estimate \eqref{unifcont} and Theorem \ref{thm:3} give the thesis.
\end{proof}

The following lower-semicontinuity result will prove to be useful in the proof of Theorem \ref{thm:4} and Corollary \ref{cor:3}. Here the integrability of $\ell$, depending on condition \eqref{psilip} below, plays a key role.

\begin{theorem}\label{thm:4+}
Consider a nonnegative convex function $\psi \colon \R^d\to[0,+\infty)$ satisfying the following condition: there exists a constant $C\ge 0$ and $1\le q <+\infty$ such that, for all $R>0$,
\begin{equation}\label{psilip}
\operatorname{Lip}(\psi, B(0, R))\le C R^{q-1}
\end{equation}
where $B(0,R)$ is the ball of radius $R$ in $\R^d$ centered at $0$. For $q$ as in \eqref{psilip}, fix $\ell \in L^q(0,T)$ and consider a sequence of functions $(f_k)_{k \in \mathbb N}$ in $\mathcal F_\ell$ converging to $f$ in the sense of \eqref{locweakconv}. Let $\mu_k:[0,T] \to \mathcal P_1(\mathbb R^{n})$ be a sequence of functions taking values
in the probability measures with finite first moment such that
\begin{equation}\label{bound}
\supp (\mu_k(t)) \Subset  \Omega,
\end{equation}
for a.e.\ $t \in [0,T]$ and $k \in \N$, where $\Omega$ is a relatively compact open set in $\mathbb R^n$. Let $\mu :[0,T] \to \mathcal P_1(\mathbb R^n)$, and assume that 
\begin{equation}\label{convwass2}
\lim_k \mathcal W_1(\mu_k(t),\mu(t)) = 0, \mbox{ for a.e.\ } t \in [0,T].
\end{equation} 
Then, we have
\begin{equation}\label{semicont+}
\liminf_{k\to +\infty}\int_0^T \langle \psi(f_k(t,\cdot)), \mu_k(t) \rangle\,dt \ge \int_0^T \langle \psi(f(t,\cdot)), \mu(t) \rangle\,dt. 
\end{equation}
\end{theorem}

\begin{remark}
For $\psi$ globally Lipschitz, as in the case of the cost functional \eqref{CSmod2}, we can simply take $q=1$.
\end{remark}

\begin{proof}
We first observe that \eqref{bound} and \eqref{convwass2} clearly imply that $\supp (\mu(t)) \Subset  \Omega$ for a.e.\ $t \in [0,T]$.
Our first goal is to show
\begin{equation}\label{semicont-}
\liminf_{k\to +\infty}\int_0^T \langle \psi(f_k(t,\cdot)), \mu(t) \rangle\,dt \ge \int_0^T \langle \psi(f(t,\cdot)), \mu(t) \rangle\,dt. 
\end{equation}
To prove this, we fix $p>n$ and consider $C_{\ell,\Om}$ as in \eqref{convesso}. \\We define the function $S^\mu \colon L^q((0,T), W^{1,p}(\Om, \R^d)) \to [0, +\infty]$ as
\begin{equation*}
S^\mu(g):=\begin{cases}
           &\displaystyle
           \int_0^T \langle \psi(g(t,\cdot)), \mu(t) \rangle\,dt\quad\hbox{if }g\in C_{\ell, \Om}\\
           &\displaystyle
           +\infty\quad\hbox{ otherwise}\,.
          \end{cases}
\end{equation*}
We want to prove that $S^\mu$ is lower semicontinuous with respect to the weak convergence of\\ $L^q((0,T), W^{1,p}(\Om, \R^d))$: with this, \eqref{semicont-} easily follows.
By convexity of $\psi$ and $C_{\ell, \Om}$, it is immediate to show that $S^\mu$ is convex. It then suffices to prove that it is lower semicontinuous in the strong topology of $L^q((0,T), W^{1,p}(\Om, \R^d))$ to obtain weak lower semicontinuity. To this end, take a sequence $g_k \in L^q((0,T), W^{1,p}(\Om, \R^d))$ strongly converging to $g$. The only relevant case is when $g_k \in C_{\ell, \Om}$, so that also $g \in C_{\ell, \Om}$. In such a case, we clearly have by \eqref{convesso} and \eqref{psilip} that there exists a constant $C'$ only depending on $C$ and the diameter of $\Om$ such that
$$
|\psi(g_k(t,x,v))-\psi(g(t,x,v))|\le (C'\ell(t))^{q-1} |g_k(t,x,v)-g(t,x,v)|
$$
for a.e.\ $t\in [0,T]$ and all $(x,v)\in \Om$. Denoting with $M$ the continuity constant of Morrey's embedding, we then get
\begin{eqnarray*}
&\displaystyle
|S^\mu(g_k)-S^\mu(g)|=\Big|\int_0^T \langle \psi(g_k(t,\cdot))-\psi(g(t,\cdot)), \mu(t) \rangle\,dt\Big|\\
&\displaystyle
\le \int_0^T \|\psi(g_k(t))-\psi(g(t))\|_{L^\infty(\Om)}\,dt \le C'^{q-1}\int_0^T \ell(t)^{q-1}\|g_k(t)-g(t)\|_{L^\infty(\Om,\R^d)}\,dt\\
&\displaystyle
\le MC'^{q-1}\int_0^T  \ell(t)^{q-1} \|g_k(t)-g(t)\|_{W^{1,p}(\Om,\R^d)}\,dt\,.
\end{eqnarray*}
Now, $\ell(t)^{q-1}$ belongs to $L^{q'}(0,T)$ by $q$-integrability of $\ell$, while $\|g_k(t)-g(t)\|_{W^{1,p}(\Om,\R^d)}$ is going to $0$ as $k\to +\infty$ in $L^q(0,T)$. Therefore, H\"older inequality implies
$$
|S^\mu(g_k)-S^\mu(g)|\to 0\,,
$$
thus \eqref{semicont-} is proved.

We claim now that
\begin{equation}\label{vanish}
\lim_{k\to +\infty}\int_0^T \langle \psi(f_k(t,\cdot)), \mu_k(t)-\mu(t) \rangle\,dt=0\,.
\end{equation}
Indeed, since $f_k \in \mathcal F_\ell$ and using \eqref{psilip}, we have that there exists a constant $C'$ only depending on $C$ and on the diameter of $\Om$ such that
$$
\operatorname{Lip}(\psi \circ f_k(t), \Om)\le C'^{q-1}\ell(t)^q
$$
for all $k\in \N$ and a.e.\ $t\in (0,T)$, where $\psi \circ f_k(t)$ is the composition of the functions $\psi$ and $f_k(t)$. Therefore
\begin{eqnarray*}
&\displaystyle
\Big|\int_0^T \langle \psi(f_k(t,\cdot)), \mu_k(t)-\mu(t) \rangle\,dt\Big|=\Big|\int_0^T \langle \psi(f_k(t,\cdot))-\psi(f_k(t,0)), \mu_k(t)-\mu(t) \rangle\,dt\Big|\\
&\displaystyle
\le C'^{q-1}\int_0^T \ell(t)^q \mathcal W_1(\mu_k(t), \mu(t)\,dt\,.
\end{eqnarray*}
By \eqref{convwass2}, this latter integrand is pointwise vanishing. Since clearly condition \eqref{bound} implies \eqref{unifmom}, by \eqref{boundw1} $\mathcal W_1(\mu_k(t), \mu(t)$ is bounded uniformly with respect to $t$ and $k$; since $\ell \in L^q(0,T)$, we can then use the dominated convergence theorem 
$$
\int_0^T \ell(t)^q \mathcal W_1(\mu_k(t), \mu(t)\,dt\to 0\,.
$$
With this, \eqref{vanish} follows: combining it with \eqref{semicont-}, we get \eqref{semicont+}.
\end{proof}

As a concluding remark for this section, we point out that, although we consider here controls in the class $\mathcal F_\ell$ depending on both the variables $x$ and $v$, our analysis apply without any change to controls that depend only on some specific variables. This could be justified by some modeling reasons: for instance, it would be consistent with the previous work \cite{CFPT} to take controls which depend exclusively by the velocity state.
Indeed, any subclass of $\mathcal F$ consisting of controls only depending on some specific variables is closed with respect to the convergence in \eqref{locweakconv}, as we clarify in the following remark.

\begin{remark}
For $1\le d<n$ we write the generic point $z$ of $\Rn$ as $z=(u,w)$,
$u \in \R^{n-d}$, $w\in \mathbb R^d$. Given $\mathcal F_\ell$ as in Definition \ref{def:admcontr}, we also introduce the following subclass of $\mathcal F_\ell$ of admissible controls given by
\begin{equation}\label{def2}
\mathcal F^w_\ell := \{ f(t,z) \in \mathcal F_\ell: f(t,z) = f(t,w) \}
\end{equation}
Trivially, $f \in \mathcal F_\ell^w$ if and only if $|f(t,0)| + \|\nabla_w f(t,\cdot)\|_{L^\infty(\mathbb R^d)} \leq \ell(t)$ for almost
every $t \in [0,T]$. We can show that $\mathcal F_\ell^w$ is closed  with respect to the weak convergence in Corollary \ref{cor:1}. Indeed, if $(f_j)_{j \in \mathbb N} \in \mathcal F_\ell^w$ is a sequence, then $(f_j)_{j \in \mathbb N}$ is also a sequence in $\mathcal F_\ell$. Assume now that $f_j$ is converging to $f \in \mathcal F_\ell$ in the sense of \eqref{locweakconv}. Let $\{u_k, k \in \mathbb N \}$ and $\{w_k, k \in \mathbb N \}$ be two countable dense subsets of $\R^{n-d}$ and $\R^d$, respectively. Pick two different $u_{k_1}$, $u_{k_2}$ in the first countable subset, and one fixed $w_{k_3}$ in the second one. Fix $t_1$ and $t_2$ in $[0, T]$ and set 
$$
\psi(t):=\chi_{(t_1, t_2)}(\delta_{u_{k_1}}-\delta_{u_{k_2}})\otimes \delta_{w_{k_3}}\,.
$$ 
Since $f_j \in \mathcal F_\ell^w$ we have for a.e.\ $t\in (t_1, t_2)$ that 
$$
\langle f_j(t), \psi(t)\rangle= f_j(t, u_{k_1}, w_{k_3})-f_j(t, u_{k_2}, w_{k_3})=f_j(t, w_{k_3})-f_j(t,w_{k_3})=0,
$$ 
so that \eqref{locweakconv} specifies easily to 
$$
0=\int_{t_1}^{t_2}[f(t,u_{k_1},w_{k_3}) - f(t,u_{k_2},w_{k_3})]\,dt
$$
By the Lebesgue theorem, we can find a set $N \subset [0,T]$ of zero Lebesgue measure and independent of the elements $u_{k_1}$, $u_{k_2}$, and $w_{k_3}$ such that
$$
0=f(t,u_{k_1},w_{k_3}) - f(t,u_{k_2},w_{k_3}),
$$
for all $t \in [0,T]\setminus N$. Since $f(t,\cdot, \cdot)$ is continuous for almost every $t$, by a density argument we infer
$$
f(t, u_1, w)-f(t, u_2, w)=0
$$
for all $t \in [0,T]\setminus N$, $u_1$ and  $u_2\in \R^{n-d}$, and $w$ in $\R^d$. This is exactly saying that $f\in \mathcal F_\ell^w$.
\end{remark}

\section{The Finite Dimensional Control Problem}\label{sec:finite}

In the following we consider problems in the phase space $\mathbb R^n$ where $n=2d$ with state variables $z=(x,v)$, $x$, $v\in \R^d$.
We state the following assumptions:
\begin{itemize}
\item[(H)] Let $H\colon \R^{2d} \to \R^{d}$ be a locally Lipschitz function such that
\begin{equation}\label{lingrowth}
|H(z)|\le C(1+|z|), \quad \mbox{for all } z \in \mathbb R^{2d};
\end{equation}
\item[(L)]  Let $L: \mathbb R^{2d} \times \mathcal P_1 (\mathbb R^{2d})\to \mathbb R_+$
be a continuous function in the state variables $(x,v)$ and such that if $(\mu_j)_{j \in \mathbb N} \subset \mathcal P_1 (\mathbb R^{2d})$ is a sequence converging narrowly to $\mu$ in  $\mathcal P_1 (\mathbb R^{2d})$, then $L(x,v,\mu_j) \to  L(x,v,\mu)$ uniformly with respect to $(x,v)$  on compact sets of $\mathbb R^{2d}$;
\item[($\Psi$)] Let $\psi: \R^d \to [0,+\infty)$ be a nonnegative convex function satisfying the following assumption: there exist $C\ge 0$ and $1\le q<+\infty$ such that
\begin{equation}\label{psilip2}
\operatorname{Lip}(\psi, B(0, R))\le C R^{q-1}
\end{equation}
for all $R>0$.
\end{itemize}
These assumptions are useful in this section, but we shall recall them also in Section \ref{sec:meanfield}, where they play again a crucial role.

We fix $q$ so that \eqref{psilip2} holds, and a function $\ell \in L^q(0,T)$. Given $N\in \mathbb N$ and an initial datum $(x_1(0), \dots, x_N(0), v_1(0), \dots, v_N(0)) \in (\mathbb R^d)^N \times (\mathbb R^d)^N$, 
we consider the following optimal control problem:
\begin{equation}\label{mainmod3}
\min_{f \in \mathcal F_\ell} \int_{0}^T \int_{\mathbb R^{2d}} \left [ L(x,v,\mu_N(t,x,v)) + \psi(f(t,x,v) \right ] d \mu_N(t,x,v) dt,
\end{equation}
where 
$$
\mu_N(t,x,v) = \frac{1}{N} \sum_{j=1}^N \delta_{(x_i(t),v_i(t))}(x,v),
$$
is the time dependent empirical atomic measure supported on the phase space trajectories $(x_i(t),v_i(t)) \in \mathbb R^{2 d}$, for $i=1,\dots N$, constrained by being the solution 
of the system
\begin{equation}\label{mainmod4}
\left \{
\begin{array}{ll}
\dot x_i = v_i, & \\
\dot v_i = ( H \star \mu_N)(x_i,v_i) + f(t,x_i, v_i), & i=1,\dots N, \quad t \in [0,T],
\end{array}
\right.
\end{equation}
 The symbol $\star$ indicates the convolution operator of a function with respect to a measure.
Let us stress that the existence of  Carath{\'e}odory solutions of \eqref{mainmod4} for any given $f \in \mathcal F_\ell$ is actually ensured by Theorem \ref{cara2} and Theorem \ref{cara-global} recalled
in the Appendix.
We start with a trajectory confinement result.

\begin{lemma}\label{lem:conf}
Let $f \in \mathcal F_\ell$ and $(x(t),v(t))$ be the solution of \eqref{mainmod4} with initial datum $(x(0),v(0))$. Then
\begin{eqnarray}\label{confeq}
\mathcal V(t) \leq \left \{ \mathcal V(0) + [1+ \mathcal X(0)]\left (2CT+\int_0^t \ell(s) ds \right)\right \} e^{(1+T) \int_0^t [2C + \ell(s)] ds},
\end{eqnarray}
for all $t \in [0,T]$, where  $\mathcal V(t) = \max_{i=1,\dots, N} |v_i(t)|$ and $\mathcal X(t)= \max_{i=1,\dots,N} | x_i(t)|$. Hence the trajectory $(x(t),v(t))$ is uniformly bounded with respect to $t \in [0,T]$, independently of the number $N$ of particles.
\end{lemma}
\begin{proof}
By integration of \eqref{mainmod4} and the linear growths of both $H$ and control function $f \in \mathcal F_\ell$, we obtain
\begin{equation}\label{eqfirst}
\mathcal V(t) \leq \mathcal V(0) + \int_0^t [2C(1+\mathcal X(s) + \mathcal V(s)) + \ell(s)(1+ \mathcal X(s)+ \mathcal V(s))] ds.
\end{equation}
Moreover, for any $0\leq s \leq t$
$$
|x_i(s)| \leq |x_i(0)| + \int_0^s |v_i(r)| dr \leq |x_i(0)| + \int_0^t |v_i(r)| dr,
$$
which combined with \eqref{eqfirst} gives
$$
\mathcal V(t) \leq \left \{ \mathcal V(0) + [1+ \mathcal X(0)] \left (2CT+\int_0^t \ell(s) ds \right) \right \}+ (1+T)\int_0^t [2C + \ell(s)] \mathcal V(s) ds.
$$
Gronwall's lemma eventually yields \eqref{confeq}.
\end{proof}
From the uniform support bound on $\mu_N$ provided by Lemma \ref{lem:conf} and from Theorems \ref{thm:3} and \ref{thm:4+} we deduce the following well-posedness result.

\begin{theorem}\label{thm:4}
The finite horizon optimal control problem \eqref{mainmod3}-\eqref{mainmod4} with initial datum $(x(0),v(0))$ has solutions.
\end{theorem}
\begin{proof}
Let us consider a minimizing sequence $(f_j)_{j \in \mathbb N}$ in $\mathcal F_\ell$ and its subsequence
$(f_{j_k})_{k \in \mathbb N}$ and $f \in \mathcal F_\ell$ as in Corollary \ref{cor:1}. For simplicity we rename $f_k = f_{j_k}$.
Let us also fix $(x^k(t),v^k(t))$ the trajectory solutions to \eqref{mainmod4} corresponding to $f_k$. 

As we have from \eqref{mainmod4} that
$$
\max_{i=1,\dots ,N} |\dot v_i^k(t)| \leq 2 C(1+ \max_{i=1,\dots, N} |x_i^k(t)| +  \max_{i=1,\dots, N} |v_i^k(t)|) + \ell(t)(1+ \max_{i=1\dots N} |v_i^k(t)|),
$$
 Lemma \ref{lem:conf} implies the equi-integrability of $\dot v_i^k(t)$, the equi-boundedness and the equi-absolute continuity of $v_i^k(t)$, hence the equi-Lispchitzianity of $x_i^k(t)$ as well, uniformly with 
respect to $k$, for all $i=1,\dots,N$. By Ascoli-Arzel\`a theorem there exist a subsequence, again renamed $(x^k(t),v^k(t))_{k \in \mathbb N}$
and an absolutely continuous trajectory $(x(t),v(t))$ in $[0,T]$ such that, for $k \to \infty$ 
\begin{equation}
\left \{
\begin{array}{ll}
x_i^k \rightrightarrows x_i, & \mbox{ in } [0,T], \quad \mbox{ for all } i=1,\dots,N,\\
v_i^k \rightrightarrows v_i, & \mbox{ in } [0,T], \quad \mbox{ for all } i=1,\dots,N,\\
\dot x_i^k \rightrightarrows \dot x_i, & \mbox{ in } [0,T], \quad \mbox{ for all } i=1,\dots,N,\\
\dot v_i^k \rightharpoonup \dot v_i, & \mbox{ in } L^q(0,T), \quad \mbox{ for all } i=1,\dots,N.\\
\end{array}
\right .
\end{equation}
Let us remark that the $\dot v_i^k \rightharpoonup \dot v_i$ in $L^q(0,T)$ can be seen again as a consequence of the more general Theorem \ref{thm:0}.
In particular $\dot x_i(t) = v_i(t)$ in $[0,T]$, for all $i=1,\dots,N$. 
Let us denote
$$
\mu_N^k = \frac{1}{N} \sum_{i=1}^N \delta_{(x_i^k(t),v_i^k(t))} \quad \mbox{and }  \mu_N = \frac{1}{N} \sum_{i=1}^N \delta_{(x_i(t),v_i(t))}.
$$
As a consequence of Lemma \ref{lem:conf} and of the uniform convergence of the trajectories we have that for every $i=1,\dots,N$
\begin{equation}\label{Wasser}
\mathcal W_1 (\delta_{(x_i^k(t),v_i^k(t))}, \delta_{(x_i(t),v_i(t))})\to 0\quad\hbox{and}\quad \mathcal W_1 (\mu_N^k(t), \mu_N(t))\to 0
\end{equation}
as $k\to +\infty$, uniformly in $t\in [0,T]$. Furthermore there exists a relatively compact open subset $\Om \subset \R^{2d}$ such that
\begin{equation}\label{confin}
\supp (\mu_N^k(t))\, \cup \,\supp (\mu_N(t)) \Subset  \Omega,
\end{equation}
for all $t \in [0,T]$ and $k\in \N$. By the linear growth of $H$ we deduce
\begin{equation}\label{potr}
H \star \mu_N^k(x_i^k,v_i^k) \rightrightarrows  H \star \mu_N(x_i,v_i), \mbox{ in } [0,T], \quad \mbox{ for } k \to \infty,
\end{equation}
see also Lemma \ref{secstim} in the Appendix. 

To prove that $(x(t),v(t))$ is actually the Carath{\'e}odory solution of \eqref{mainmod4} associated to 
the forcing term $f$, we therefore have only to show that for all $i=1,\dots,N$ one has
\begin{equation}\label{claim}
\dot v_i = ( H \star \mu_N)(x_i,v_i) + f(t,x_i, v_i).
\end{equation}
This is clearly equivalent to the following: for every $\xi \in \R^d$ and every $\hat t\in [0,T]$ it holds:
\begin{equation}\label{claim2}
\xi\cdot \int_0^{\hat t}\dot v_i(t)\,dt = \xi\cdot \int_0^{\hat t}[( H \star \mu_N)(x_i(t),v_i(t))+ f(t,x_i(t), v_i(t))]\,dt.
\end{equation}
In view of the weak $L^q$ convergence of $\dot v_i^k$ to $\dot v_i$ and of \eqref{potr}, \eqref{claim2} actually reduces to show
$$
\lim_{k\to +\infty}\xi\cdot \int_0^{\hat t}f(t,x_i^k(t), v_i^k(t))\,dt=\xi\cdot \int_0^{\hat t}f(t,x_i(t), v_i(t))\,dt\,.
$$
Given $\varphi_{\xi} \in C^1_c(\R^{2d}, \R^d)$ such that $\varphi_{\xi}\equiv \xi$ in $\Om$, the above equality reduces to
$$
\lim_{k\to +\infty}\int_0^{\hat t}\langle \varphi_{\xi},f(t,\cdot)\,\delta_{(x_i^k(t),v_i^k(t))}\rangle\,dt=\int_0^{\hat t}\langle \varphi_{\xi},f(t,\cdot)\,\delta_{(x_i(t),v_i(t))}\rangle\,dt
$$
which follows from \eqref{compmu}, \eqref{Wasser}, and \eqref{confin}.

Using \eqref{Wasser} and \eqref{confin}, the inequality
\begin{equation}\label{lowersemi}
\liminf_{k\to +\infty}\int_{0}^T \int_{\mathbb R^{2d}} \psi(f_k(t,x,v))\, d \mu_N^k(t,x,v) dt \ge \int_{0}^T \int_{\mathbb R^{2d}} \psi(f(t,x,v))\, d \mu_N(t,x,v) dt
\end{equation}
follows now directly from \eqref{semicont+}. Furthermore, condition (L) yields
\begin{eqnarray}
\lim_k \int_{0}^T \int_{\mathbb R^{2d}} L(x,v,\mu_N^k) d\mu_N^k(t,x,v) dt &=& \lim_k \int_{0}^T \frac{1}{N} \sum_{i=1}^N  L(x_i^k(t),v_i^k(t),\mu_N^k(t)) dt \nonumber\\
&=& \int_{0}^T \frac{1}{N} \sum_{i=1}^N  L(x_i(t),v_i(t),\mu_N(t)) dt \nonumber \\
&=& \int_{0}^T \int_{\mathbb R^{2d}} L(x,v,\mu_N) d\mu_N(t,x,v) dt. \label{contin}
\end{eqnarray}
Combining \eqref{lowersemi} and \eqref{contin} gives the optimality of $f$ as a minimizer of \eqref{mainmod3} under the solution constraint \eqref{mainmod4}.
\end{proof}

\section{Mean-Field Solutions}\label{sec:meanfield}

In this section we are concerned with the limit for $N \to \infty$ of the solutions of the ODE system \eqref{mainmod} to solutions of the PDE of the type \eqref{PDEmod}. First we need to define a proper concept of
solution for \eqref{PDEmod}. To this aim, assuming the $L^1$ integrability of the function $\ell$ defining the class $\mathcal F_\ell$ is sufficient.

\begin{definition}\label{defz}
Let $\ell \in L^1(0,T)$. Fix a function $f$ belonging to the class ${\mathcal F}_\ell$. Given a locally Lipschitz function $H\colon \R^{2d} \to \R^{d}$ satisfying \eqref{lingrowth}, we say that a map $\mu\colon[0,T]\to \PP(\R^{2d})$ continuous with respect to $\WW$ is a weak equi-compactly supported solution of the equation
\begin{equation}\label{PDEmodXX}
\frac{\partial \mu}{\partial t} + v \cdot \nabla_x \mu = \nabla_v \cdot \left [ \left ( H \star \mu + f \right ) \mu \right ],
\end{equation}
with forcing term $f$ on the interval $[0, T]$ if there exists $R>0$ such that
\begin{equation}\label{hypo2++}
{\rm supp }\,\mu(t)\subset B(0,R),
\end{equation}
for every $t \in [0, T]$, and, defining $w_{H, \mu, f}(t,x,v) \colon [0, T]\times \R^d \times \R^d \to \R^d\times \R^d$ as
$$
w_{H, \mu, f}(t,x,v):=(v, H\star \mu(t)(x,v)+ f(t,x,v))
$$
one has
\begin{equation}\label{solution}
\frac{d}{dt}\int_{\R^{2d}}\zeta(x,v)\,d\mu(t)(x,v)= \int_{\R^{2d}}\nabla \zeta(x,v)\cdot w_{H, \mu, f}(t,x,v)\,d\mu(t)(x,v)
\end{equation}
for every $\zeta \in C^\infty_c(\R^d \times \R^d)$, in the sense of distributions.
\end{definition}

Since the linear span of functions of the type $\eta(t) \zeta(x,v)$, with $\eta \in C^\infty_c(0,T)$ and $\zeta \in C^\infty_c(\R^d \times \R^d)$ is dense in $C^1_c((0, T)\times \R^d \times \R^d)$ it is not difficult to see that actually, \eqref{solution} is equivalent to saying that
$$
\int_0^T\int_{\R^{2d}}\big(\partial_t \varphi (t,x,v)+\nabla_{x,v} \varphi(t,x,v)\cdot w_{H, \mu, f}(t,x,v)\big)\,d\mu(t)(x,v)\,dt=0,
$$
for every $\varphi \in C^1_c((0, T)\times \R^d \times \R^d)$.
Using also the explicit expression of $w_{H, \mu, f}$ we can equivalently reformulate \eqref{solution} by asking that the equality
\begin{equation}\label{solution2}
\begin{array}{c}
\displaystyle
\int_{\R^{2d}}\zeta(x,v)\,d\mu(t)(x,v)-\int_{\R^{2d}}\zeta(x,v)\,d\mu(0)(x,v)= \\[5pt]
\displaystyle
\int_{\R^{2d}}\big(\nabla_x \zeta(x,v)\cdot v+\nabla_v\zeta(x,v)\cdot H\star \mu(t)(x,v)+ \nabla_v\zeta(x,v)\cdot f(t,x,v)\big)\,d\mu(t)(x,v),
\end{array}
\end{equation}
holds for every $t \in [0, T]$ and every $\zeta \in C^\infty_c(\R^d \times \R^d)$, in the sense of distributions. Finally, we can also consider test functions depending also on $t$, and defining solutions as satisfying the equality
\begin{eqnarray}
0&=&\!\!\int_0^T\!\!\int_{\R^{2d}}\big(\partial_t \varphi (t,x,v)+\nabla_x \varphi (t,x,v)\cdot v+\nabla_v \varphi (t,x,v)\cdot H\star \mu(t)(x,v) \\ \nonumber
&& \phantom{XXXXXXXXXXXXXXXXXXXXXx}+ \nabla_v \varphi (t,x,v)\cdot f(t,x,v)\big)\,d\mu(t)(x,v),\label{solution3}
\end{eqnarray}
for every $\varphi \in C^1_c((0, T)\times \R^d \times \R^d)$: again, this is equivalent to \eqref{solution}.

\begin{remark}
Observe that $w_{H, \mu, f}(t,x,v)$ is a Carath{\'e}odory vector field, thus measurable with respect to any product measure of the type $\Lone \times \mu$ with $\Lone$ the Lebesgue measure on $[0, T]$ and $\mu$ a Borel probability measure on $\R^d\times \R^d$. Furthermore, by \eqref{lingrowth} and \eqref{hypo2++}, we can show that \eqref{trzu} in the Appendix holds for $n=2d$ and $p=d$. Taking also into account \eqref{lipt}  of Lemma \ref{stimesceme} for $n=2d$ and $p=d$ and \eqref{terr} in the Appendix, it follows that the $W^{1, \infty}$ norm of $w_{H, \mu, f}(t,\cdot,\cdot)$ is bounded by an $L^1$-function of $t$ on any compact subset of $\R^d \times \R^d$. Jointly with the aforementioned measurability property, this allows us to repeat the arguments in \cite[Section 8.1]{AGS} proving that $\mu(t)$ is a weak equi-compactly supported solution of \eqref{PDEmodXX} with forcing term $f$ on the interval $[0, T]$ if and only if it satisfies \eqref{hypo2++} and the measure-theoretical fixed point equation
\begin{equation}\label{pushf}
\mu(t)=({\mathcal T}^\mu_t)_\sharp\mu_0,
\end{equation}
with $\mu_0:=\mu(0)$ and ${\mathcal T}^\mu_t$ is the flow function defined by \eqref{definitflow} in the Appendix. Here $({\mathcal T}^\mu_t)_\sharp$ denotes the push-forward of $\mu_0$ through ${\mathcal T}^\mu_t$.
\end{remark}

\subsection{Empirical equi-compactly supported solutions}

Let us now again fix $\ell \in L^1(0,T)$ and a locally Lipschitz function $H\colon \R^{2d} \to \R^{d}$ satisfying \eqref{lingrowth}. Given $N\in \mathbb N$, an initial datum $(x_1^0, \dots, x_N^0, v_1^0, \dots, v_N^0) \in B(0,R_0)\subset (\mathbb R^d)^N \times (\mathbb R^d)^N$, with $R_0>0$ independent of $N$,
and a function $f_N \in \mathcal F_\ell$, perhaps also depending on $N$, we consider the following time dependent empirical atomic measure
\begin{equation}\label{empsolN}
\mu_N(t,x,v) = \frac{1}{N} \sum_{j=1}^N \delta_{(x_i(t),v_i(t))}(x,v),
\end{equation}
supported on the phase space trajectories $(x_i(t),v_i(t)) \in \mathbb R^{2 d}$, for $i=1,\dots N$, defining the solution 
of the system
\begin{equation}\label{Nmodel}
\left \{
\begin{array}{ll}
\dot x_i = v_i, & \\
\dot v_i = ( H \star \mu_N)(x_i,v_i) + f_N(t,x_i,v_i), & i=1,\dots N, \quad t \in [0,T],
\end{array}
\right.
\end{equation}
with initial datum $(x(0),v(0))=(x^0,v^0)$.
We remark again that under our assumptions on $f$ and $H$, existence and uniqueness of Carath{\'e}odory solutions of \eqref{Nmodel} are ensured (see Appendix for details) hence the well-posedness of \eqref{empsolN}.
\begin{lemma}\label{empsol}
Under the assumptions considered since the beginning of this section, let us define $\mu_N$ as in \eqref{empsolN}. The the following properties hold:
\begin{itemize}
\item[(a)] $\supp(\mu_N(t)) \subset B(0,R_T)$, where $R_T>0$ is independent of $N$;
\item[(b)] $\mathcal W_1(\mu_N(t_1), \mu_N(t_2)) \leq L_T |t_1 - t_2|$, for all $t_1,t_2 \in [0,T]$, for a suitable constant $L_T>0$ dependent on $T$ and independent of $N$;
\item[(c)] for all $t \in [0,T]$ and for all $\zeta \in C_c^1(\mathbb R^{2d})$
$$
\langle \zeta, \mu_N(t) - \mu_N(0) \rangle = \int_0^t \left [ \int_{\mathbb R^{2d}} \nabla \zeta(x,v) \cdot w_{H,\mu_N, f_N}(t,x,v) d \mu_N(t,x,v) \right ] dt.
$$
\end{itemize}
In particular the maps $t \to \mu_N(t)$ are equi-compactly supported solutions of the equation \eqref{PDEmodXX} in the sense of Definition \ref{defz} for all $N \in \mathbb N$.
\end{lemma}
\begin{proof}
The property (a) is a direct consequence of the equi-boundedness of the datum \\$(x_1(0), \dots, x_N(0), v_1(0), \dots, v_N(0)) \in B(0,R_0)$, and of Lemma \ref{lem:conf}. Let us prove (b). As the measures $\mu_N(t_1)$ and $\mu_N(t_2)$ are actually atomic, in this case the $1$-Wasserstein distance can be expressed in terms of the
$\ell^1$-norm of the Euclidean distances of
the respective supporting atoms:
$$
\mathcal W_1(\mu_N(t_1), \mu_N(t_2)) = \frac{1}{N} \sum_{i=1}^N (|x_i(t_1) - x_i(t_2)| + |v_i(t_1) - v_i(t_2)| ).
$$ 
In the following $L_T$ may be a different constant at different places, but always  dependent on $T$, and independent of $N$. 
As $\dot x_i(t) = v_i(t)$, from (a) we know already that $|x_i(t_1) - x_i(t_2)|\leq L_T |t_1 - t_2|$ for all $i=1,\dots N$, for a suitable constant $L_T>0$. 
As we have
$$
\max_{i=1,\dots ,N} |\dot v_i(t)| \leq 2 C \left (1+ \max_{i=1,\dots, N} |x_i(t)| +  \max_{i=1,\dots, N} |v_i(t)| \right ) + \ell(t)(1+ \max_{i=1,\dots, N} |x_i(t)|+\max_{i=1\dots N} |v_i(t)|),
$$
 Lemma \ref{lem:conf} implies the equi-integrability of $\dot v_i(t)$, whence also $|v_i(t_1) - v_i(t_2)|\leq L_T |t_1 - t_2|$ for all $i=1,\dots N$, for a suitable constant $L_T>0$.
The validity of (c) is a standard argument, which is developed by considering the differentiation
\begin{eqnarray*} \frac{d}{dt} \langle \zeta, \mu_N(t)\rangle &=&  \frac{1}{N} \frac{d}{dt} \sum_{i=1}^N \zeta(x_i(t),v_i(t))\\
&=& \frac{1}{N} \left [ \sum_{i=1}^N \nabla_x \zeta(x_i(t),v_i(t)) \cdot \dot x_i(t) + \sum_{i=1}^N \nabla_v \zeta(x_i(t),v_i(t)) \cdot \dot v_i(t) \right],
\end{eqnarray*}
and directly applying the substitutions as in \eqref{Nmodel}.
\end{proof}
\subsection{Convergence of empirical solutions}
In this section we show how solutions to \eqref{Nmodel} converges to solutions of \eqref{PDEmodXX} in the sense of Definition \ref{defz}.
\begin{theorem}\label{thm:5}
Let us consider a sequence of equi-compactly supported empirical probability measures $(\mu_N^0)_{N \in \mathbb N}$, where $\mu_N^0 = \frac{1}{N} \sum_{i=1}^M \delta_{((x_N^0)_i,(v_N^0)_i)}(x,v) $ for suitable sets of points $(x^0_N,v_N^0)$ in $(\mathbb R^d)^N \times (\mathbb R^d)^N$ with the properties
\begin{itemize}
\item[(i)] $\mu_N^0$ has support equi-bounded in $\mathbb R^{2d}$, i.e., $(x^0_N,v_N^0)\in B(0,R_0)$, for $R_0>0$ independent of $N$;
\item[(ii)] there exists a compactly supported $\mu_0 \in \mathcal P_1(\mathbb R^{2d})$ such that $\lim_{N \to \infty} \mathcal W_1(\mu_N^0,\mu^0)=0$.
\end{itemize}
Given $\ell \in L^1(0,T)$ and $(f_N)_{N \in \mathbb N}$ an arbitrary sequence in $\mathcal F_\ell$, we can accordingly define $\mu_N(t)$ as the empirical equi-compactly supported solutions \eqref{empsolN} of \eqref{Nmodel} with initial data $(x^0_N,v_N^0)$ in $(\mathbb R^d)^N \times (\mathbb R^d)^N$ and forcing term $f_N$, respectively, for all $N \in \mathbb N$. \\
Then there exist a subsequence $(f_{N_k})_{k \in \mathbb N}$ converging in the sense of \eqref{locweakconv} to a function $f \in \mathcal F_\ell$, and a subsequence $(\mu_{N_k}(t))_{k \in \mathbb N}$ converging in Wasserstein distance uniformly in $t \in [0,T]$  to a weak equi-compactly supported solution $\mu(t)$ of equation \eqref{PDEmodXX} with forcing term $f$ and initial condition
$\mu(0) = \mu^0$  in the sense of Definition \ref{defz}.
\end{theorem}
\begin{proof}
Thanks to the equi-boundedness assumption (i) we can apply Lemma \ref{empsol} (a) and (b), and the sequence of measures $(\mu_N(t))_{N \in \mathbb N}$ is actually equi-bounded and equi-Lipschitz. By an application of the Ascoli-Arzel\`a theorem for functions on $[0,T]$ with values in the complete metric space $(\mathcal P_1({B(0,R_T)}), \mathcal W_1)$ (here $\mathcal P_1(B(0,R_T))$ actually denotes the space of probability measures compactly supported in $B(0,R_T)$,  endowed with the $1$-Wasserstein distance); up to extracting a subsequence, we deduce the existence of a uniform limit $\mu(t)$ in $1$-Wasserstein distance, which is actually equi-supported in $B(0,R_T)$, uniformly with respect to $t \in [0,T]$, for $R_T>0$ as in  Lemma \ref{empsol} (a).
Obviously it is also $\mu(0) = \mu^0$, and by lower-semicontinuity of the Wasserstein distance with respect to the narrow convergence we have
\begin{equation}\label{preprima}
\mathcal W_1(\mu(t_2), \mu(t_1)) \leq L_T |t_1- t_2|,
\end{equation}
for all $t_1, t_2 \in [0,T]$, where $L_T>0$ is as in Lemma \ref{empsol} (b). Moreover,
\begin{equation}\label{prima}
\lim_{k \to \infty} \langle \zeta, \mu_{N_k}(t) -\mu_{N_k}(0) \rangle =  \langle \zeta, \mu(t) -\mu(0) \rangle,
\end{equation}
for all $\zeta \in C_c^1(\mathbb R^{2d})$. By possibly extracting an additional subsequence and by applying Theorem \ref{thm:3} and Corollary \ref{cor:2}, we immediately obtain
\begin{equation}\label{seconda}
\lim_{k \to \infty} \int_0^t \int_{\mathbb R^{2d}} (\nabla_v \zeta(x,v) \cdot f_{N_k}(t,x,v)) d \mu_{N_k}(t,x,v) dt =  \int_0^t \int_{\mathbb R^{2d}} (\nabla_v \zeta(x,v) \cdot f(t,x,v)) d \mu(t,x,v) dt,
\end{equation}
and by weak-$*$ convergence and the dominated convergence theorem
\begin{equation}\label{terza}
\lim_{k \to \infty} \int_0^t \int_{\mathbb R^{2d}} (\nabla_v \zeta(x,v) \cdot v ) d \mu_{N_k}(t,x,v) dt =  \int_0^t \int_{\mathbb R^{2d}} (\nabla_v \zeta(x,v) \cdot v) d \mu(t,x,v) dt,
\end{equation}
for all $\zeta \in C_c^1(\mathbb R^{2d})$. By Lemma \ref{secstim} in Appendix we also obtain that for every $\rho >0$ 
$$
\lim_{k \to \infty}  \|H\star \mu_{N_k}(t)-H\star \mu(t)\|_{L^\infty(B(0,\rho)} =0,
$$
and, as $\zeta \in C_c^1(\mathbb R^{2d})$ has compact support, it follows that
$$
\lim_{k \to \infty}  \| \nabla_v \zeta \cdot (H\star \mu_{N_k}(t) - H\star \mu(t)) \|_\infty =0.
$$
Since the product measures $\mathcal L^1\llcorner_{[0,T]} \times \frac{1}{T} \mu_{N_k}(t)$ converge in $\mathcal P_1([0,T] \times \mathbb R^{2d})$ to $\mathcal L^1\llcorner_{[0,T]} \times \frac{1}{T} \mu(t)$, we obtain also
\begin{equation}\label{quarta}
\lim_{k \to \infty} \int_0^t \int_{\mathbb R^{2d}} (\nabla_v \zeta(x,v) \cdot H\star \mu_{N_k}(t) ) d \mu_{N_k}(t,x,v) dt =  \int_0^t \int_{\mathbb R^{2d}} (\nabla_v \zeta(x,v) \cdot H\star \mu(t)) d \mu(t,x,v) dt,
\end{equation}
The statement now follows by combining \eqref{preprima}, \eqref{prima}, \eqref{seconda}, \eqref{terza}, and \eqref{quarta}.
\end{proof}
We now prove the lower semicontinuity of the cost functionals with respect to the convergences in the previous theorem, what can be seen as a $\Gamma-\liminf$ condition.
\begin{corollary}\label{cor:3}
Let us now  assume as in (L) that $L: \mathbb R^{2d} \times \mathcal P_1 (\mathbb R^{2d})\to \mathbb R_+$
be a continuous function in the state variables $(x,v)$ and such that if $(\mu_j)_{j \in \mathbb N} \subset \mathcal P_1 (\mathbb R^{2d})$ is a sequence converging narrowly to $\mu$ in  $\mathcal P_1 (\mathbb R^{2d})$, then $L(x,v,\mu_j) \to  L(x,v,\mu)$ uniformly with respect to $(x,v)$  on compact sets of $\mathbb R^{2d}$. Consider a nonnegative convex function $\psi:\R^d \to [0,+\infty)$ satisfying condition ($\Psi$). Besides the assumptions of Theorem \ref{thm:5}, suppose that $\ell \in L^q(0,T)$, with $1\le q<+\infty$ being as in \eqref{psilip2}. We then have the following lower-semicontinuity property 
\begin{eqnarray}
 \liminf_{k \to \infty} && \int_0^T \int_{\mathbb R^{2d}} (L(x,v,\mu_{N_k}) + \psi(f_{N_k}(t,x,v))) d\mu_{N_k}(t,x,v)  \nonumber \\ 
&\geq & \int_0^T \int_{\mathbb R^{2d}} (L(x,v,\mu) + \psi(f(t,x,v))) d\mu(t,x,v)dt, \label{Gammaliminf}
\end{eqnarray}
where $\mu_{N_k}$ and $f_{N_k}$ are the elements of the subsequences of the statement of Theorem \ref{thm:5}.
\end{corollary}
\begin{proof}
As $\mu_{N_k}$ and $\mu$ are compactly supported within a ball $B(0,R_T)$, for $R_T>0$, by assumption (L) and the dominated convergence theorem, we can conclude that
\begin{equation}\label{unifconvL}
\lim_{k \to \infty} \int_0^T \int_{\mathbb R^{2d}} L(x,v,\mu_{N_k}) d\mu_{N_k}(t,x,v)  dt = \int_0^T \int_{\mathbb R^{2d}} L(x,v,\mu)d\mu(t,x,v)  dt.
\end{equation}
due to the uniform convergence of $\mu_{N_k}(t)$ to $\mu(t)$ in Wasserstein distance $\mathcal W_1$, as shown in the proof of Theorem \ref{thm:5}.
Since $\mu_{N_k}(t)$ are equi-compactly supported,  by this uniform convergence and Theorem \ref{thm:4+} we get the lower-semicontinuity of the second term:
\begin{equation}\label{lowersemcontf}
\liminf_{k\to \infty} \int_0^T \int_{\mathbb R^{2d}} \psi(f_{N_k}(t,x,v)) d\mu_{N_k}(t,x,v) dt \geq \int_0^T \int_{\mathbb R^{2d}} \psi(f(t,x,v)) d\mu(t,x,v) dt.
\end{equation}
By combining \eqref{unifconvL} and \eqref{lowersemcontf}, we eventually show \eqref{Gammaliminf}.
\end{proof}

\subsection{Existence of solutions}

With very similar arguments as the ones we used to prove Theorem \ref{thm:5} and Corollary \ref{cor:3} we obtain the following existence result, with additional limit property of the cost functional. This can be seen as a $\Ga$-limsup condition.

\begin{theorem}\label{thm:6} 
Assume that we are given maps $H$, $L$, and $\psi$ as in assumptions (H), (L), and ($\Psi$) of Section \ref{sec:finite}. For $1\le q<+\infty$ so that \eqref{psilip2} holds, let $\ell(t)$ be a fixed function in $L^q(0,T)$.
Let $\mu^0 \in \mathcal P_1(\mathbb R^{2d})$ be a given probability measure with compact support and $f \in \mathcal F_\ell$ a forcing term. We assume that the sequence
$(\mu^0_N)_{N \in \mathbb N}$ of atomic empirical measures  $\mu^0_N = \frac{1}{N} \sum_{i=1}^N \delta_{(x_N^0)_i,(v_N^0)_i}(x,v)$ is such that $\lim_{N \to \infty} \mathcal W_1(\mu^0_N,\mu^0)=0$. Let 
\begin{equation}\label{empsolNN}
\mu_N(t,x,v) = \frac{1}{N} \sum_{j=1}^N \delta_{(x_i(t),v_i(t))}(x,v),
\end{equation}
be supported on the phase space trajectories $(x_i(t),v_i(t)) \in \mathbb R^{2 d}$, for $i=1,\dots N$, defining the solution 
of the system
\begin{equation}\label{NNmodel}
\left \{
\begin{array}{ll}
\dot x_i = v_i, & \\
\dot v_i = ( H \star \mu_N)(x_i,v_i) + f(t,x_i,v_i), & i=1,\dots N, \quad t \in [0,T],
\end{array}
\right.
\end{equation}
with initial datum $(x(0),v(0))=(x_N^0,v_N^0)$. Then there exists  a map $\mu:[0,T] \to \mathcal P_1(\mathbb R^{2d})$ such that
\begin{itemize}
\item[(i)] $\lim_{N \to \infty} \mathcal W_1(\mu_{N}(t), \mu(t))=0$ uniformly with respect to $t \in [0,T]$;
\item[(ii)] $\mu$ is a weak equi-compactly supported solution of \eqref{PDEmodXX} with forcing term $f$ in the sense of Definition \ref{defz};
\item[(iii)] the following limit holds:
\begin{eqnarray}
 \lim_{N \to \infty} && \int_0^T \int_{\mathbb R^{2d}} (L(x,v,\mu_{N}) + \psi(f(t,x,v))) d\mu_{N}(t,x,v)  \nonumber \\ 
&= & \int_0^T \int_{\mathbb R^{2d}} (L(x,v,\mu) + \psi(f(t,x,v))) d\mu(t,x,v)dt. \label{Gammalimsup}
\end{eqnarray}
\end{itemize}

\end{theorem}
\begin{proof}
The results (i) and (ii) can be shown precisely as done in the proof of Theorem \ref{thm:5}. The only additional note is 
that, due to Theorem \ref{uniq} in Appendix, $\mu$ is actually the unique weak solution of \eqref{PDEmodXX}, hence 
the whole sequence $(\mu_N)_{N \in \mathbb N}$ converges to $\mu$ and not only a subsequence. By uniform convergence of $\mu_N$ to $\mu$ in $1$-Wasserstein distance, and by \eqref{vanish} with $f_k$ constantly equal to $f$ we have
$$
\lim_{N \to \infty} \int_0^T \int_{\mathbb R^{2d}}\psi(f(t,x,v)) d\mu_{N}(t,x,v) = \int_0^T \int_{\mathbb R^{2d}}  \psi(f(t,x,v)) d\mu(t,x,v)dt
$$
while, since $\mu_N$ and $\mu$ are compactly supported within a ball $B(0,R_T)$, by assumption (L) and the dominated convergence theorem, we have
\begin{equation*}
\lim_{N \to \infty} \int_0^T \int_{\mathbb R^{2d}} L(x,v,\mu_N) d\mu_N(t,x,v)  dt = \int_0^T \int_{\mathbb R^{2d}} L(x,v,\mu)d\mu(t,x,v)  dt.
\end{equation*}
The limit \eqref{Gammalimsup} follows.
\end{proof}

\section{Mean-Field Optimal Control}
We are now able to state the main result of this paper, which is summarizing all the findings we obtained so far, in particular combining the concepts of
mean-field and $\Gamma$-limits.

\begin{theorem}\label{thm:7}
Assume that we are given maps $H$, $L$, and $\psi$ as in assumptions (H), (L), and ($\Psi$) of Section \ref{sec:finite}. For $1\le q<+\infty$ so that \eqref{psilip2} holds, let $\ell(t)$ be a fixed function in $L^q(0,T)$. For $N\in \mathbb N$ and an initial datum $((x_N^0)_1, \dots, (x_N^0)_N, (v_N^0)_1, \dots, (v_N^0)_N) \in B(0,R_0) \subset (\mathbb R^d)^N \times (\mathbb R^d)^N$,
for $R_0>0$ independent of $N$,  we consider the following finite dimensional optimal control problem
\begin{equation}\label{mainmod3N}
\min_{f \in \mathcal F_\ell} \int_{0}^T \int_{\mathbb R^{2d}} \left [ L(x,v,\mu_N(t,x,v)) + \psi( f(t,x,v)) \right ] d \mu_N(t,x,v) dt,
\end{equation}
where 
$$
\mu_N(t,x,v) = \frac{1}{N} \sum_{j=1}^N \delta_{(x_i(t),v_i(t))}(x,v),
$$
is the time dependent empirical atomic measure supported on the phase space trajectories $(x_i(t),v_i(t)) \in \mathbb R^{2 d}$, for $i=1,\dots N$, constrained by being the solution 
of the system
\begin{equation}\label{mainmod4N}
\left \{
\begin{array}{ll}
\dot x_i = v_i, & \\
\dot v_i = ( H \star \mu_N)(x_i,v_i) + f(t,x_i,v_i), & i=1,\dots N, \quad t \in [0,T],
\end{array}
\right.
\end{equation}
with initial datum $(x(0),v(0))=(x_N^0,v_N^0)$ and, for consistency, we set $$\mu_N^0 = \frac{1}{N} \sum_{i=1}^M \delta_{((x_N^0)_i,(v_N^0)_i)}(x,v).$$
For all $N \in \mathbb N$ let us denote the function $f_N \in \mathcal F_\ell$ as a solution of the finite dimensional optimal control problem \eqref{mainmod3N}-\eqref{mainmod4N}.
\\
If there exists a  compactly supported $\mu_0 \in \mathcal P_1(\mathbb R^{2d})$ such that $\lim_{N \to \infty} \mathcal W_1(\mu_N^0,\mu^0)=0$, then there exists a
subsequence $(f_{N_k})_{k \in \mathbb N}$ and a function $f_\infty \in \mathcal F_\ell$ such that $f_{N_k}$ converges to $f_\infty$ in the sense of \eqref{locweakconv} and $f_\infty$ is a solution of the infinite dimensional optimal control problem
\begin{equation}\label{mainmod3oo}
\min_{f \in \mathcal F_\ell} \int_{0}^T \int_{\mathbb R^{2d}} \left [ L(x,v,\mu(t,x,v)) + \psi(f(t,x,v)) \right ] d \mu(t,x,v) dt,
\end{equation}
where $\mu:[0,T]  \to \mathcal P_1(\mathbb R^{2d})$ is the unique weak solution of 
\begin{equation}\label{PDEmodXXoo}
\frac{\partial \mu}{\partial t} + v \cdot \nabla_x \mu = \nabla_v \cdot \left [ \left ( H \star \mu + f \right ) \mu \right ],
\end{equation}
with initial datum $\mu(0):=\mu^0$ and forcing term $f$, in the sense of Definition \ref{defz}.
\end{theorem}
\begin{proof}
Let us first of all notice that the existence of an optimal control $f_N$ for the finite dimensional optimal control problem \eqref{mainmod3N}-\eqref{mainmod4N} is ensured by Theorem \ref{thm:4}.
Let $g$ an arbitrary function in $\mathcal F_\ell$ and $\mu_g$ be the corresponding solution to \eqref{PDEmodXXoo} with datum $\mu_g(0):=\mu^0$, which exists thanks to Theorem \ref{thm:6} and whose uniqueness follows from Theorem \ref{uniq} in Appendix.
We also fix the  sequence $(\mu_g)_{N}$ of atomic measures uniformly converging to $\mu_g$ as in Theorem \ref{thm:6}.
We consider now the converging subsequence $(f_{N_k})_{k \in \mathbb N}$ considered in Theorem \ref{thm:5} and we denote $f_\infty$ its limit in the sense of \eqref{locweakconv}. We further denote
with $\mu_{\infty}$ the corresponding solution to \eqref{PDEmodXXoo}, when the forcing term is $f_\infty$. Then, by lower-semicontinuity as in Corollary \ref{cor:3} and minimality of $f_{N_k}$ 
\begin{eqnarray*}
&&\int_0^T \int_{\mathbb R^{2d}} (L(x,v,\mu_{\infty}) + \psi(f_\infty(t,x,v))) d\mu_{\infty}(t,x,v) dt \\
&\leq& \liminf_{k \to \infty} \int_0^T \int_{\mathbb R^{2d}} (L(x,v,\mu_{N_k}) + \psi(f_{N_k}(t,x,v))) d\mu_{N_k}(t,x,v) dt\\
&\leq& \liminf_{k \to \infty} \int_0^T \int_{\mathbb R^{2d}} (L(x,v,(\mu_g)_{N_k}) + \psi(g(t,x,v))) d(\mu_g)_{N_k}(t,x,v) dt\\
&=& \int_0^T \int_{\mathbb R^{2d}} (L(x,v,\mu_g) + \psi(g(t,x,v))) d\mu_g(t,x,v) dt,
\end{eqnarray*}
where the last equality follows from \eqref{Gammalimsup} in Theorem \ref{thm:6}. By arbitrariness of $g$, we conclude that $f_\infty$ in an optimal control for the problem \eqref{mainmod3oo}-\eqref{PDEmodXXoo}.
\end{proof}

\section{Appendix}
For the reader's convenience we start by briefly recalling some well-known results about solutions to Carath{\'e}odory differential equations. We fix an interval $[0,T]$ on the real line, and let $n\ge 1$.
Given a domain $\Om \subset \R^n$, a Carath{\'e}odory function $g\colon[0,T]\times \Om \to \R^n$, and $0<\tau \le T$, a function $y\colon [0,\tau]\to \Om$ is called a solution of the Carath{\'e}odory differential equation
\begin{equation}\label{cara}
\dot y(t)=g(t, y(t))
\end{equation}
on $[0,\tau]$ if and only if $y$ is absolutely continuous and \eqref{cara} is satisfied a.e.\ in $[0,\tau]$.
The following existence and uniqueness result holds.
\begin{theorem}\label{cara2}
Consider an interval $[0,T]$ on the real line, a domain $\Om \subset \R^n$, $n\ge 1$, and a Carath{\'e}odory function $g\colon[0,T]\times \Om \to \R^n$. Assume that there exists a function $m \in L^1((0,T))$ such that
$$
|g(t,y)|\le m(t)
$$
for a.e.\ $t \in [0,T]$ and every $y \in \Om$. Then, given $y_0 \in \Om$, there exists $0<\tau \le T$ and a solution $y(t)$ of \eqref{cara} on $[0,\tau]$ satisfying $y(0)=y_0$. 

If in addition there exists a function $l \in L^1((0,T))$ such that
\begin{equation}\label{cara3}
|g(t,y_1)-g(t, y_2)|\le l(t)|y_1-y_2|
\end{equation}
for a.e.\ $t \in [0,T]$ and every $y_1$, $y_2 \in \Om$, the solution is uniquely determined on $[0,\tau]$ by the initial condition $y_0$.
\end{theorem}

\begin{proof}
See, for instance, \cite[Chapter 1, Theorems 1 and 2]{Fil}.
\end{proof}

Also the global existence theorem and a Gronwall estimate on the solutions can be easily generalized to this setting.

\begin{theorem}\label{cara-global}
Consider an interval $[0,T]$ on the real line and a Carath{\'e}odory function $g\colon[0,T]\times \R^n \to \R^n$. Assume that there exists a function $m \in L^1((0,T))$ such that
\begin{equation}\label{ttz}
|g(t,y)|\le m(t)(1+|y|)
\end{equation}
for a.e.\ $t \in [0,T]$ and every $y \in \R^n$. Then, given $y_0 \in \R^n$, there exists a solution $y(t)$ of \eqref{cara} defined on the whole interval $[0,T]$ which satisfies $y(0)=y_0$. Any solution satisfies
\begin{equation}\label{gron}
|y(t)|\le \Big(|y_0|+ \int_0^t m(s)\,ds\Big) \,e^{\int_0^t m(s)\,ds}
\end{equation}
for every $t \in [0,T]$.

If in addition, for every relatively compact open subset of $\R^n$, \eqref{cara3} holds, the solution is uniquely determined on $[0,T]$ by the initial condition $y_0$.
\end{theorem}

\begin{proof}
Let $C_0:= (|y_0|+ \int_0^T m(s)\,ds) \,e^{\int_0^T m(s)\,ds}$. Take a ball $\Om \subset \R^p$  centered at $0$ with radius strictly greater than $C_0$. Existence of a local solution defined on an interval $[0,\tau]$ and taking values in $\Om$ follows now easily from \eqref{ttz} and Theorem \ref{cara2}. Using \eqref{ttz}, any solution of \eqref{cara} with initial datum $y_0$ satisfies
$$
|y(t)|\le |y_0|+ \int_0^t m(s)\,ds+\int_0^t m(s)|y(s)|\,ds
$$
for every $t \in [0,\tau]$, therefore \eqref{gron} follows from Gronwall's Lemma. In particular the graph of a solution $y(t)$ cannot reach the boundary of $[0,T]\times \Om$ unless $\tau=T$, therefore existence of a global solution follows for instance from \cite[Chapter 1, Theorem 4]{Fil}. If \eqref{cara3} holds, uniqueness of the global solution follows from Theorem \ref{cara2}.
\end{proof}

The usual results on continuous dependence on the data hold also in this setting: in particular, we will use this Lemma, following from \eqref{gron} and the Gronwall inequality.

\begin{lemma}
Let $g_1$ and $g_2\colon[0,T]\times \R^n \to \R^n$ be Carath{\'e}odory functions both satisfying \eqref{ttz} for a  function $m \in L^1(0,T)$. Let $r>0$ and define 
$$
\rho_{r, m, T}:=\Big(r+ \int_0^T m(s)\,ds\Big) \,e^{\int_0^T m(s)\,ds}\,.
$$ 
Assume in addition that there exists a function $l\in L^1(0,T)$
$$
|g_1(t, y_1)-g_1(t, y_2)|\le l(t)|y_1-y_2|
$$
for every $t \in [0, T]$ and every $y_1$, $y_2$ such that $|y_i|\le \rho_{r, m, T}$, $i=1,2$.
Set
$$
q(t):=\|g_1(t, \cdot)-g_2(t, \cdot)\|_{L^\infty(B(0, \rho_{r, m, T}))}\,.
$$
Then, if $\dot y_1(t)=g(t, y_1(t))$, $\dot y_2(t)=g_2(t, y_2(t))$, $|y_1(0)|\le r$ and $|y_2(0)|\le r$, one has
\begin{equation}\label{gronvalla}
|y_1(t)-y_2(t)|\le e^{\int_0^t l(s)\,ds}|y_1(0)-y_2(0)|+\int_0^t e^{\int_s^t l(\sigma)\,d\sigma}q(s)\,ds
\end{equation}
for every $t \in [0, T]$.
\end{lemma}

We will need the following Lemma. In its statement, we recall that $\PP(\R^n)$ denotes the space of probability measures on $\R^n$ with finite first moment. This is a metric space when endowed with the Wasserstein distance $\WW$.

\begin{lemma}\label{stimesceme}
Let $H\colon \R^n \to \R^{p}$, $n\ge p \ge 1$ be a locally Lipschitz function such that
\begin{equation}\label{hypo}
|H(y)|\le C(1+|y|), \quad \mbox{ for all } y \in \mathbb R^n, 
\end{equation}
and $\mu\colon[0,T]\to \PP(\R^n)$ be a continuous map with respect to $\WW$. Then there exists a constant $C'$ such that
\begin{equation}\label{trzu}
|H\star \mu(t) (y)|\le C'(1+|y|),
\end{equation}
for every $t \in [0, T]$ and every $y \in \R^n$. Furthermore, if
\begin{equation}\label{hypo2}
{\rm supp }\,\mu(t)\subset B(0,R),
\end{equation}
for every $t \in [0, T]$, then for every compact subset $K$ of $\R^n$ there exists a constant $L_{R,K}$ such that
\begin{equation}\label{lipt}
|H\star \mu(t) (y_1)-H\star \mu(t) (y_2)|\le L_{R,K}|y_1-y_2|,
\end{equation}
for every $t \in [0, T]$ and every $y_1$, $y_2 \in K$.
\end{lemma}

\begin{proof}
One has from \eqref{hypo}
$$
|H\star \mu(t) (y)|=\Big|\int_{\R^n}H(y-\xi)\,d\mu(t)(\xi)\Big|\le C(1+|y|)+C\int_{\R^n}|\xi|\,d\mu(t)(\xi)\,;
$$
since $\int_{\R^n}|\xi|\,d\mu(t)(\xi)$ is uniformly bounded on $[0,T]$ by our continuity assumption, \eqref{trzu} follows.

If \eqref{hypo2} holds, then for every $y_1,y_2 \in K$ one has
$$
|H\star \mu(t) (y_1)-H\star \mu(t) (y_2)|\le \int_{B(0,R)} \left |H(y_1-\xi) - H(y_2-\xi) \right |\,d\mu(t)(\xi) \leq L_{R,K} |y_1-y_2|.
$$
\end{proof}

We now fix a dimension $d \ge 1$ and consider the system of ODE's on $\R^{2d}$
\begin{equation}\label{system}
\begin{cases}
\dot X(t)=V(t)\\
\dot V(t)= H\star \mu(t) (X(t), V(t))+ f(t, X(t), V(t))
\end{cases}
\end{equation}
on an interval $[0, T]$. Here $X, V$ are both mappings from $[0, T]$ to $\R^d$, $H\colon \R^{2d} \to \R^{d}$ is a locally Lipschitz function satisfying \eqref{lingrowth}, $\mu\colon[0,T]\to \PP(\R^n)$ is a continuous map with respect to $\WW$ satisfying \eqref{hypo2} and $f$ belongs to the class ${\mathcal F}_\ell$ defined in \eqref{def2} for a fixed function $\ell \in L^1(0,T)$. In particular, we have
\begin{equation}\label{terr}
|f(t, X, V)|\le \ell(t)(1+|(X, V)|)
\end{equation}
for every $V \in \R^d$. It follows then from these assumptions and Lemma \ref{stimesceme} that all the hypothesis of Theorem \ref{cara-global} are satisfied. Therefore, however given $P_0:=(X_0, V_0)$ in $\R^{2d}$ there exists a unique solution $P(t):=(X(t), V(t))$ to \eqref{system} with initial datum $P_0$ defined on the whole interval $[0, T]$. We can therefore consider the family of flow maps ${\mathcal T}^\mu_t \colon \R^{2d} \to \R^{2d}$ indexed by $t\in [0,T]$ and defined by
\begin{equation}\label{definitflow}
{\mathcal T}^\mu_t(P_0):=P(t)
\end{equation}
where $P(t)$ is the value of the unique solution to \eqref{system} starting from $P_0$ at time $t=0$. The notation aims also at stressing the dependence of these flow maps on the given mapping $\mu(t)$. 
We can easily recover, as consequence of \eqref{gronvalla}, similar estimates as in \cite[Lemmas 3.7 and 3.8]{CanCarRos10}: we report the statement and a sketch of the proof of this result to allow the reader to keep track of the dependence of these constants on the data of the problem.

\begin{lemma}
Let $H\colon \R^{2d} \to \R^{d}$ be a locally Lipschitz function satisfying \eqref{lingrowth}, let $\mu\colon[0,T]\to \PP(\R^{2d})$ and $\nu\colon[0,T]\to \PP(\R^{2d})$ be continuous maps with respect to $\WW$ both satisfying 
\begin{equation}\label{hypo2+}
{\rm supp }\,\mu(t)\subset B(0,R)\quad\hbox{and}\quad {\rm supp }\,\nu(t)\subset B(0,R)
\end{equation}
for every $t \in [0, T]$.
Consider $f$ belonging to the class ${\mathcal F}_\ell$ introduced in Definition \label{class}, for a fixed function $\ell \in L^1(0,T)$, and the flow maps ${\mathcal T}^\mu_t$ and ${\mathcal T}^\nu_t$ associated to the system \eqref{system} and to the system
\begin{equation}\label{system+}
\begin{cases}
\dot X(t)=V(t)\\
\dot V(t)= H\star \nu(t) (X(t), V(t))+ f(t, X(t), V(t))\,,
\end{cases}
\end{equation}
respectively, on $[0, T]$. Let $C$ be the constant in \eqref{lingrowth}. Fix $r>0$: then there exist a constant $\rho$ and a function $l\in L^1(0,T)$, both depending only on $r$, $C$, $R$,  $\ell$, and $T$ such that
\begin{equation}\label{contflow}
|{\mathcal T}^\mu_t(P_1)-{\mathcal T}^\nu_t(P_2)|\le e^{\int_0^t l(s)\,ds}|P_1-P_2|+ \int_0^t e^{\int_s^t l(\sigma)\,d\sigma}\|H\star \mu(s)-H\star \nu(s)\|_{L^\infty(B(0,\rho))}\,ds
\end{equation}
whenever $|P_1|\le r$ and $|P_2|\le r$, for every $t \in [0, T]$.
\end{lemma}

\begin{proof}
Let $g_1$ and $g_2\colon [0,T]\times \R^{2d}\to \R^{2d}$  be the right-hand sides of \eqref{system}, and \eqref{system+}, respectively. As in \eqref{trzu} we can find a constant $C'$ which depends only on $C$ and $R$ such that
\begin{equation}\label{trzu+}
|H\star \mu(t) (P)|\le C'(1+|P|)\quad\hbox{and}\quad|H\star \nu(t) (P)|\le C'(1+|P|)
\end{equation}
for every $t \in [0, T]$ and every $P \in \R^{2d}$. Setting now $\hat m(t)=1+C'+\ell(t)$ and also using \eqref{terr}, it follows that $g_1$ and $g_2$ both satisfy \eqref{ttz} with $m(t)$ replaced by $\hat m(t)$. Therefore, for every $P_1$ and $P_2 \in \R^{2d}$ such that $|P_i|\le r$, $i=1,2$ and every $t\in [0,T]$, \eqref{gron} gives
$$
|{\mathcal T}^\mu_t(P_1)|\le \Big(r+ \int_0^T \hat m(s)\,ds\Big) \,e^{\int_0^T \hat m(s)\,ds}\quad\hbox{and}\quad|{\mathcal T}^\nu_t(P_2)|\le \Big(r+ \int_0^T \hat m(s)\,ds\Big) \,e^{\int_0^T \hat m(s)\,ds}\,.
$$
Set $\rho:=\Big(r+ \int_0^T \hat m(s)\,ds\Big) \,e^{\int_0^T \hat m(s)\,ds}$. Now, obviously
$$
\|g_1(t,\cdot)-g_2(t,\cdot)\|_{L^\infty(B(0,\rho))}=\|H\star \mu(t)-H\star \nu(t)\|_{L^\infty(B(0,\rho))}
$$
for every $t \in [0, T]$. Furthermore, by \eqref{lipt}, the definition of $\rho$, and since $f$ belongs to ${\mathcal F}_\ell$, the Lipschitz constant of $g_1(t, \cdot)$  on $B(0,\rho)$ can be estimated for a.e. \ $t \in [0, T]$  with a function $l(t)\in L^1(0,T)$ only depending on $\ell(t)$, $R$, $C$, $r$ and $T$. With this, the conclusion follows at once from \eqref{gronvalla}.
\end{proof}

We will use \eqref{pushf} to prove uniqueness and stability of equi-compactly supported solutions of \eqref{solution}. We recall the following two Lemmata, both proved in \cite{CanCarRos10}.

\begin{lemma}\label{primstim}
Let $E_1$ and $E_2 \colon \Rn \to \Rn$ be two bounded Borel measurable functions. Then, for every $\mu \in \PP(\Rn)$ one has
\begin{equation*}
\WW((E_1)_\sharp \mu, (E_2)_\sharp \mu) \le \|E_1-E_2\|_{L^\infty({\rm supp}\,\mu)}\,.
\end{equation*}
If in addition $E_1$ is locally Lipschitz continuous, and $\mu$, $\nu \in \PP(\Rn)$ are both compactly supported on a ball $B_r$ of $\Rn$, then
\begin{equation}\label{transport}
\WW((E_1)_\sharp \mu, (E_1)_\sharp \nu) \le L_r \WW(\mu, \nu)\,,
\end{equation}
where $L_r$ is the Lipschitz constant of $E_1$ on $B_r$.
\end{lemma}

\begin{proof}
For the sake of the reader, we sketch only the proof of \eqref{transport} as it does not appear exactly equally reported in \cite[Lemmata 3.11 and 3.15]{CanCarRos10}. Let $\tilde \pi$ be the optimal transfer plan for $(E_1)_\sharp \mu$ and 
$(E_1)_\sharp \nu$ and $\pi$ the one of $\mu$ and $\nu$. Then
\begin{eqnarray*}
&& \mathcal W_1((E_1)_\sharp \mu, (E_1)_\sharp \nu) = \int_{\mathbb R^n \times \mathbb R^n} |x-y| d\tilde \pi(x,y) \leq \int_{\mathbb R^n \times \mathbb R^n} |x-y| d((E_1 \times E_1)_\sharp ¸\pi)(x,y) \\
&=& \int_{\mathbb R^n \times \mathbb R^n} |E_1(x)-E_1(y)| d \pi(x,y) = L_r \int_{\mathbb R^n \times \mathbb R^n} |x-y| d \pi(x,y) =L_r \mathcal W_1(\mu,\nu).
\end{eqnarray*}
\end{proof}

\begin{lemma}\label{secstim}
Let $H\colon \R^{2d} \to \R^{d}$ be a locally Lipschitz function satisfying \eqref{lingrowth}, let $\mu\colon[0,T]\to \PP(\R^{2d})$ and $\nu\colon[0,T]\to \PP(\R^{2d})$ be continuous maps with respect to $\WW$ both satisfying 
\begin{equation*}
{\rm supp }\,\mu(t)\subset B(0,R)\quad\hbox{and}\quad {\rm supp }\,\nu(t)\subset B(0,R)
\end{equation*}
for every $t \in [0, T]$. Then for every $\rho >0$ there exists a constant $L_{\varrho, R}$ such that
$$
\|H\star \mu(t)-H\star \nu(t)\|_{L^\infty(B(0,\rho)}\le L_{\varrho, R}\WW(\mu(t), \nu(t))
$$
for every $t \in [0, T]$.
\end{lemma}

\begin{proof}
See \cite[Lemma 4.7]{CanCarRos10}.
\end{proof}

With the previous Lemmata and \eqref{contflow}, we can easily prove the following result.
\begin{theorem}\label{uniq}
Fix a function  $f$ belonging to the class ${\mathcal F}_\ell$ for a given $\ell \in L^1(0,T)$. Consider a locally Lipschitz function $H\colon \R^{2d} \to \R^{d}$ satisfying \eqref{lingrowth} with a constant $C$. Fix $T>0$  and let $\mu(t)$ and $\nu(t)$ be two equi-compactly supported solutions  of \eqref{PDEmodXX} with forcing term $f$ on the interval $[0, T]$. Let $\mu_0:=\mu(0)$ and $\nu_0:=\nu(0)$. Consider $r>0$ such that
$$
{\rm supp}\,\mu_0 \subset B(0, r)\quad\hbox{and}\quad {\rm supp}\,\nu_0 \subset B(0, r)
$$
and $R>0$ such that
\begin{equation}\label{supptot}
{\rm supp}\,\mu(t) \subset B(0, R)\quad\hbox{and}\quad {\rm supp}\,\nu(t) \subset B(0, R)
\end{equation}
for every $t \in[0, T]$. Then, there exist a function $\delta \in L^1(0,T)$ depending only on $r$, $C$, $R$,  $\ell$, and $T$ such that
\begin{equation}\label{stab}
\WW(\mu(t), \nu(t)) \le e^{\int_0^t \delta(s)\,ds} \WW(\mu_0, \nu_0)
\end{equation}
for every $t \in [0, T]$. In particular, equi-compactly supported solutions of \eqref{solution} are uniquely determined by the initial datum.
\end{theorem}

\begin{proof}
Let  ${\mathcal T}^\mu_t$ and ${\mathcal T}^\nu_t$ be the flow maps associated to the system \eqref{system} and to the system \eqref{system+}, respectively.
By \eqref{pushf}, the triangle inequality, and Lemma \ref{primstim} we have for every $t$
\begin{equation}\label{start}
\begin{array}{c}
\WW(\mu(t), \nu(t))=\WW(({\mathcal T}^\mu_t)_\sharp \mu_0, ({\mathcal T}^\nu_t)_\sharp \nu_0)  \\[5pt]
\le \WW(({\mathcal T}^\mu_t)_\sharp \mu_0, ({\mathcal T}^\mu_t)_\sharp \nu_0) + \WW(({\mathcal T}^\mu_t)_\sharp \nu_0, ({\mathcal T}^\nu_t)_\sharp \nu_0)\le L_r \WW(\mu_0, \nu_0)+\|{\mathcal T}^\mu_t-{\mathcal T}^\nu_t\|_{L^\infty(B(0,r))}
\end{array}
\end{equation}
where $L_r$ is the Lipschitz constant of ${\mathcal T}^\mu_t$ on the ball $B(0,r)$.

Using \eqref{contflow} with $\mu=\nu$ we get that there exists a function $l\in L^1(0,T)$ only depending on $r$, $C$, $R$,  $\ell$, and $T$ such that
\begin{equation}\label{stima1}
L_r \le e^{\int_0^t l(s)\,ds}\,.
\end{equation}
Again by \eqref{contflow} with $P_1= P_2$ there exist a constant $\rho$ and an $L^1$ function, still denoted by $l$, both depending only on $r$, $C$, $R$,  $\ell$, and $T$ such that
\begin{equation}\label{stima2}
\|{\mathcal T}^\mu_t-{\mathcal T}^\nu_t\|_{L^\infty(B(0,r))}\le \int_0^t e^{\int_s^t l(\sigma)\,d\sigma}\|H\star \mu(s)-H\star \nu(s)\|_{L^\infty(B(0,\rho))}\,ds\,.
\end{equation}

Combining \eqref{start}, \eqref{stima1}, and \eqref{stima2} with Lemma \ref{secstim}, we get the existence of an $L^1$ function, still denoted by $l(t)$, and of a constant $L$, both depending only on $r$, $C$, $R$,  $\ell$, and $T$ such that
$$
\WW(\mu(t), \nu(t))\le e^{\int_0^t l(s)\,ds} \,\WW(\mu_0, \nu_0)+ L \int_0^t e^{\int_s^t l(\sigma)\,d\sigma}\, \WW(\mu(s), \nu(s)) \,ds
$$
for every $t \in [0, T]$, or equivalently
$$
e^{-{\int_0^t l(s)\,ds}}\,\WW(\mu(t), \nu(t))\le  \WW(\mu_0, \nu_0)+ L \int_0^t e^{-{\int_0^s l(\sigma)\,d\sigma}}\, \WW(\mu(s), \nu(s)) \,ds\,.
$$
The Gronwall inequality gives now
$$
e^{-{\int_0^t l(s)\,ds}}\,\WW(\mu(t), \nu(t))\le  \WW(\mu_0, \nu_0)e^{Lt}
$$
which is exactly \eqref{stab} with $\delta(t)= L+l(t)$.
\end{proof}

\begin{remark}
The existence result of Theorem \ref{thm:6} gives an explicit estimate of an $R$ satisfying \eqref{supptot}, once the constants $r$, $C$, and $T$, and the function $\ell$ appearing in the statement of Theorem \ref{uniq} are given. As a byproduct of uniqueness, the function $\delta$ in \eqref{stab} is therefore only depending on $r$, $C$, $\ell$, and $T$.
\end{remark}

\section*{Acknowledgement}

Massimo Fornasier  acknowledges the support of the ERC-Starting Grant HDSPCONTR ``High-Dimensional Sparse Optimal Control''. 
Francesco Solombrino acknowledges the hospitality of the Johann Radon Institute for Computational and Applied Mathematics (RICAM) of the Austrian Academy of Sciences 
during the early preparation of this work.
\bibliographystyle{abbrv}
\bibliography{biblioflock}

\begin{thebibliography}{10}

\bibitem{AFP00}
L.~Ambrosio, N.~Fusco, and D.~Pallara.
\newblock {\em {Functions of Bounded Variation and Free Discontinuity
  Problems.}}
\newblock {Oxford: Clarendon Press}, 2000.

\bibitem{AGS}
L.~Ambrosio, N.~Gigli, and G.~Savar{\'e}.
\newblock {\em Gradient Flows in Metric Spaces and in the Space of Probability
  Measures}.
\newblock Lectures in Mathematics ETH Z\"urich. Birkh\"auser Verlag, Basel,
  second edition, 2008.

\bibitem{BCCCCGLOPPVZ09}
M.~Ballerini, N.~Cabibbo, R.~Candelier, A.~Cavagna, E.~Cisbani, L.~Giardina,
  L.~Lecomte, A.~Orlandi, G.~Parisi, A.~Procaccini, M.~Viale, and
  V.~Zdravkovic.
\newblock Interaction ruling animal collective behavior depends on topological
  rather than metric distance: evidence from a field study.
\newblock {\em Proceedings of the National Academy of Sciences},
  105(4):1232--1237, 2008.

\bibitem{CDFSTB03}
S.~Camazine, J.~Deneubourg, N.~Franks, J.~Sneyd, G.~Theraulaz, and E.~Bonabeau.
\newblock {Self-Organization in Biological Systems}.
\newblock {\em Princeton University Press}, 2003.

\bibitem{CanCarRos10}
J.~A. Canizo, J.~A. Carrillo, and J.~Rosado.
\newblock A well-posedness theory in measures for some kinetic models of
  collective motion.
\newblock {\em Math. Mod. Meth. Appl. Sci.}, 2010.

\bibitem{CFPT}
M.~Caponigro, M.~Fornasier, B.~Piccoli, and E.~Tr\'elat.
\newblock {Sparse stabilization and control of the Cucker-Smale model}.
\newblock {\em preprint: arXiv:1210.5739}, 2012.

\bibitem{CCH13}
J.~A. Carrillo, Y.-P. Choi, and M.~Hauray.
\newblock {The derivation of swarming models: mean-field limit and Wasserstein
  distances}.
\newblock {\em preprint: arXiv:1304.5776}, 2013.

\bibitem{MR2507454}
J.~A. Carrillo, M.~R. D'Orsogna, and V.~Panferov.
\newblock Double milling in self-propelled swarms from kinetic theory.
\newblock {\em Kinet. Relat. Models}, 2(2):363--378, 2009.

\bibitem{cafotove10}
J.~A. Carrillo, M.~Fornasier, G.~Toscani, and F.~Vecil.
\newblock Particle, kinetic, and hydrodynamic models of swarming.
\newblock In G.~Naldi, L.~Pareschi, G.~Toscani, and N.~Bellomo, editors, {\em
  Mathematical Modeling of Collective Behavior in Socio-Economic and Life
  Sciences}, Modeling and Simulation in Science, Engineering and Technology,
  pages 297--336. Birkh{\"a}user Boston, 2010.

\bibitem{caclku12}
E.~Casas, C.~Clason, and K.~Kunisch.
\newblock Approximation of elliptic control problems in measure spaces with
  sparse solutions.
\newblock {\em SIAM J. Control Optim.}, 50(4):1735--1752, 2012.

\bibitem{ChuDorMarBerCha07}
Y.~Chuang, M.~D'Orsogna, D.~Marthaler, A.~Bertozzi, and L.~Chayes.
\newblock State transition and the continuum limit for the 2{D} interacting,
  self-propelled particle system.
\newblock {\em Physica D}, (232):33--47, 2007.

\bibitem{ChuHuaDorBer07}
Y.~Chuang, Y.~Huang, M.~D'Orsogna, and A.~Bertozzi.
\newblock Multi-vehicle flocking: scalability of cooperative control algorithms
  using pairwise potentials.
\newblock {\em IEEE International Conference on Robotics and Automation}, pages
  2292--2299, 2007.

\bibitem{clku11}
C.~Clason and K.~Kunisch.
\newblock A duality-based approach to elliptic control problems in
  non-reflexive {B}anach spaces.
\newblock {\em ESAIM Control Optim. Calc. Var.}, 17(1):243--266, 2011.

\bibitem{clku12}
C.~Clason and K.~Kunisch.
\newblock A measure space approach to optimal source placement.
\newblock {\em Comput. Optim. Appl.}, 53(1):155--171, 2012.

\bibitem{CouFra02}
I.~Couzin and N.~Franks.
\newblock Self-organized lane formation and optimized traffic flow in army
  ants.
\newblock {\em Proc. R. Soc. Lond.}, B 270:139--146, 2002.

\bibitem{CKFL05}
I.~Couzin, J.~Krause, N.~Franks, and S.~Levin.
\newblock Effective leadership and decision making in animal groups on the
  move.
\newblock {\em Nature}, 433:513--516, 2005.

\bibitem{crlo65}
A.~J. Craig and I.~Fl{\"u}gge-Lotz.
\newblock Investigation of optimal control with a minimum-fuel consumption
  criterion for a fourth-order plant with two control inputs; synthesis of an
  efficient suboptimal control.
\newblock {\em J. Basic Engineering}, 87:39--58, 1965.

\bibitem{crpito10}
E.~Cristiani, B.~Piccoli, and A.~Tosin.
\newblock {Modeling self-organization in pedestrians and animal groups from
  macroscopic and microscopic viewpoints.}
\newblock In G.~Naldi, L.~Pareschi, G.~Toscani, and N.~Bellomo, editors, {\em
  Mathematical Modeling of Collective Behavior in Socio-Economic and Life
  Sciences}, Modeling and Simulation in Science, Engineering and Technology.
  Birkh{\"a}user Boston, 2010.

\bibitem{crpito11}
E.~Cristiani, B.~Piccoli, and A.~Tosin.
\newblock {Multiscale modeling of granular flows with application to crowd
  dynamics.}
\newblock {\em Multiscale Model. Simul.}, 9(1):155--182, 2011.

\bibitem{CuckerDong11}
F.~Cucker and J.-G. Dong.
\newblock A general collision-avoiding flocking framework.
\newblock {\em IEEE Trans. Automat. Control}, 56(5):1124--1129, 2011.

\bibitem{cucker-mordecki}
F.~Cucker and E.~Mordecki.
\newblock Flocking in noisy environments.
\newblock {\em J. Math. Pures Appl. (9)}, 89(3):278--296, 2008.

\bibitem{CucSma07}
F.~Cucker and S.~Smale.
\newblock Emergent behavior in flocks.
\newblock {\em IEEE Trans. Automat. Control}, 52(5):852--862, 2007.

\bibitem{cusm07}
F.~Cucker and S.~Smale.
\newblock {On the mathematics of emergence.}
\newblock {\em Jpn. J. Math.}, 2(1):197--227, 2007.

\bibitem{CucSmaZho04}
F.~Cucker, S.~Smale, and D.~Zhou.
\newblock Modeling language evolution.
\newblock {\em Found. Comput. Math.}, 4(5):315--343, 2004.

\bibitem{DM}
G.~Dal~Maso.
\newblock {\em An Introduction to {$\Gamma$}-Convergence}.
\newblock Progress in Nonlinear Differential Equations and their Applications,
  8. Birkh\"auser Boston Inc., Boston, MA, 1993.

\bibitem{Fed69}
H.~Federer.
\newblock {\em {Geometric Measure Theory.}}
\newblock {Die Grundlehren der mathematischen Wissenschaften in
  Einzeldarstellungen. 153. Berlin-Heidelberg-New York: Springer-Verlag}, 1969.

\bibitem{Fil}
A.~F. Filippov.
\newblock {\em Differential equations with Discontinuous Righthand Sides},
  volume~18 of {\em Mathematics and its Applications (Soviet Series)}.
\newblock Kluwer Academic Publishers Group, Dordrecht, 1988.
\newblock Translated from the Russian.

\bibitem{glo08}
R.~Glowinski.
\newblock {\em Numerical Methods for Nonlinear Variational Problems}.
\newblock Scientific Computation. Springer-Verlag, Berlin, 2008.
\newblock Reprint of the 1984 original.

\bibitem{GC04}
G.~Gr\'egoire and H.~Chat\'e.
\newblock Onset of collective and cohesive motion.
\newblock {\em Phy. Rev. Lett.}, (92), 2004.

\bibitem{hestwa12}
R.~Herzog, G.~Stadler, and G.~Wachsmuth.
\newblock Directional sparsity in optimal control of partial differential
  equations.
\newblock {\em SIAM J. Control and Optimization}, 50(2):943--963, 2012.

\bibitem{HCM03}
M.~Huang, P.~Caines, and R.~Malham{\'e}.
\newblock {Individual and mass behaviour in large population stochastic
  wireless power control problems: centralized and Nash equilibrium solutions}.
\newblock {\em Proceedings of the 42nd IEEE Conference on Decision and Control
  Maui, Hawaii USA, December 2003}, pages 98--103, 2003.

\bibitem{MR2000132}
A.~Jadbabaie, J.~Lin, and A.~S. Morse.
\newblock Correction to: ``{C}oordination of groups of mobile autonomous agents
  using nearest neighbor rules'' [{IEEE} {T}rans. {A}utomat. {C}ontrol {\bf 48}
  (2003), no. 6, 988--1001; {MR} 1986266].
\newblock {\em IEEE Trans. Automat. Control}, 48(9):1675, 2003.

\bibitem{KeMinAuWan02}
J.~Ke, J.~Minett, C.-P. Au, and W.-Y. Wang.
\newblock Self-organization and selection in the emergence of vocabulary.
\newblock {\em Complexity}, 7:41--54, 2002.

\bibitem{kese70}
E.~F. Keller and L.~A. Segel.
\newblock {Initiation of slime mold aggregation viewed as an instability.}
\newblock {\em J. Theor. Biol.}, 26(3):399--415, 1970.

\bibitem{KocWhi98}
A.~Koch and D.~White.
\newblock The social lifestyle of myxobacteria.
\newblock {\em Bioessays 20}, pages 1030--1038, 1998.

\bibitem{lali07}
J.-M. Lasry and P.-L. Lions.
\newblock {Mean field games.}
\newblock {\em Jpn. J. Math. (3)}, 2(1):229--260, 2007.

\bibitem{LeoFio01}
N.~Leonard and E.~Fiorelli.
\newblock Virtual leaders, artificial potentials and coordinated control of
  groups.
\newblock {\em Proc. 40th IEEE Conf. Decision Contr.}, pages 2968--2973, 2001.

\bibitem{Niw94}
H.~Niwa.
\newblock Self-organizing dynamic model of fish schooling.
\newblock {\em J. Theor. Biol.}, 171:123--136, 1994.

\bibitem{NCM10}
M.~Nuorian, P.~Caines, and R.~Malham{\'e}.
\newblock Synthesis of {Cucker-Smale} type flocking via mean field stochastic
  control theory: {Nash} equilibria.
\newblock {\em Proceedings of the 48th Allerton Conf. on Comm., Cont. and
  Comp., Monticello, Illinois, pp. 814-819, Sep. 2010}, pages 814--815, 2010.

\bibitem{NCM11}
M.~Nuorian, P.~Caines, and R.~Malham{\'e}.
\newblock Mean field analysis of controlled {Cucker-Smale} type flocking:
  Linear analysis and perturbation equations.
\newblock {\em Proceedings of 18th IFAC World Congress Milano (Italy) August 28
  - September 2, 2011}, pages 4471--4476, 2011.

\bibitem{PE99}
J.~Parrish and L.~Edelstein-Keshet.
\newblock Complexity, pattern, and evolutionary trade-offs in animal
  aggregation.
\newblock {\em Science}, 294:99--101, 1999.

\bibitem{ParVisGru02}
J.~Parrish, S.~Viscido, and D.~Gruenbaum.
\newblock Self-organized fish schools: An examination of emergent properties.
\newblock {\em Biol. Bull.}, 202:296--305, 2002.

\bibitem{PerGomElo09}
L.~Perea, G.~G\'omez, and P.~Elosegui.
\newblock Extension of the {C}ucker-{S}male control law to space flight
  formations.
\newblock {\em AIAA Journal of Guidance, Control, and Dynamics}, 32:527--537,
  2009.

\bibitem{1151.82351}
B.~Perthame.
\newblock {Mathematical tools for kinetic equations.}
\newblock {\em Bull. Am. Math. Soc., New Ser.}, 41(2):205--244, 2004.

\bibitem{be07}
B.~Perthame.
\newblock {\em {Transport Equations in Biology.}}
\newblock {Basel: Birkh\"auser}, 2007.

\bibitem{pive12}
K.~Pieper and B.~Vexler.
\newblock A priori error analysis for discretization of sparse elliptic optimal
  control problems in measure space.
\newblock {\em preprint}, 2012.

\bibitem{prtrzu13}
Y.~Privat, E.~Tr{\'e}lat, and E.~Zuazua.
\newblock Complexity and regularity of maximal energy domains for the wave
  equation with fixed initial data.
\newblock {\em preprint}, 2013.

\bibitem{rave10}
R.~Rannacher and B.~Vexler.
\newblock Adaptive finite element discretization in {PDE}-based optimization.
\newblock {\em GAMM-Mitt.}, 33(2):177--193, 2010.

\bibitem{Rom96}
W.~Romey.
\newblock Individual differences make a difference in the trajectories of
  simulated schools of fish.
\newblock {\em Ecol. Model.}, 92:65--77, 1996.

\bibitem{MR2438215}
M.~B. Short, M.~R. D'Orsogna, V.~B. Pasour, G.~E. Tita, P.~J. Brantingham,
  A.~L. Bertozzi, and L.~B. Chayes.
\newblock A statistical model of criminal behavior.
\newblock {\em Math. Models Methods Appl. Sci.}, 18(suppl.):1249--1267, 2008.

\bibitem{st09}
G.~Stadler.
\newblock {Elliptic optimal control problems with $L^1$-control cost and
  applications for the placement of control devices.}
\newblock {\em Comput. Optim. Appl.}, 44(2):159--181, 2009.

\bibitem{SugSan97}
K.~Sugawara and M.~Sano.
\newblock Cooperative acceleration of task performance: {F}oraging behavior of
  interacting multi-robots system.
\newblock {\em Physica D}, 100:343--354, 1997.

\bibitem{TonTu95}
J.~Toner and Y.~Tu.
\newblock Long-range order in a two-dimensional dynamical xy model: {H}ow birds
  fly together.
\newblock {\em Phys. Rev. Lett.}, 75:4326--4329, 1995.

\bibitem{vicsek}
T.~Vicsek, A.~Czirok, E.~Ben-Jacob, I.~Cohen, and O.~Shochet.
\newblock Novel type of phase transition in a system of self-driven particles.
\newblock {\em Phys. Rev. Lett.}, 75:1226--1229, 1995.

\bibitem{viza12}
T.~Vicsek and A.~Zafeiris.
\newblock Collective motion.
\newblock {\em Physics Reports}, 517:71--140, 2012.

\bibitem{vi09}
C.~Villani.
\newblock {\em Optimal Transport}, volume 338 of {\em Grundlehren der
  Mathematischen Wissenschaften [Fundamental Principles of Mathematical
  Sciences]}.
\newblock Springer-Verlag, Berlin, 2009.
\newblock Old and new.

\bibitem{voma06}
G.~Vossen and H.~Maurer.
\newblock {$L^1$ minimization in optimal control and applications to robotics.}
\newblock {\em Optimal Control Applications and Methods}, 27:301--321, 2006.

\bibitem{wawa11}
G.~Wachsmuth and D.~Wachsmuth.
\newblock {Convergence and regularization results for optimal control problems
  with sparsity functional.}
\newblock {\em ESAIM, Control Optim. Calc. Var.}, 17(3):858--886, 2011.

\bibitem{YEECBKMS09}
C.~Yates, R.~Erban, C.~Escudero, L.~Couzin, J.~Buhl, L.~Kevrekidis, P.~Maini,
  and D.~Sumpter.
\newblock Inherent noise can facilitate coherence in collective swarm motion.
\newblock {\em Proceedings of the National Academy of Sciences},
  106:5464--5469, 2009.

\end{thebibliography}

\end{document}